\documentclass[11pt,reqno]{amsart}
\usepackage[letterpaper,margin=1in]{geometry}
\usepackage{amsmath,amssymb,booktabs,tikz}
\usepackage{placeins}
\usepackage[hidelinks]{hyperref}
\newcommand{\articlekeywords}{weak Galerkin finite elements, conforming enrichment, moment-preserving reconstruction, corner singularities, residual selection}
\hypersetup{pdftitle={Conforming Enrichment of Weak Galerkin Spaces for Elliptic Problems with Low-Regularity Solutions},pdfauthor={Chunmei Wang},pdfsubject={Conforming enrichment for weak Galerkin discretizations of elliptic problems},pdfkeywords={\articlekeywords}}
\numberwithin{equation}{section}
\newtheorem{theorem}{Theorem}[section]
\newtheorem{lemma}[theorem]{Lemma}
\newtheorem{proposition}[theorem]{Proposition}

\theoremstyle{definition}
\newtheorem{algorithm}{Algorithm}
\newcommand{\Th}{\mathcal T_h}
\newcommand{\Sh}{\mathcal T_h^{\triangle}}
\newcommand{\highform}{\mathcal K_h}
\newcommand{\jumpform}{\mathcal D_h}
\newcommand{\Vh}{V_h^0}
\newcommand{\Wm}{W_m}
\newcommand{\R}{\mathcal R}
\newcommand{\Ah}{\mathcal A_h}
\newcommand{\Eh}{E_h}
\newcommand{\Pol}{\mathbb P}
\newcommand{\N}{\mathcal N}
\newcommand{\en}[1]{\lVert #1\rVert_a}
\newcommand{\hn}[1]{\lVert #1\rVert_{\Ah}}
\newcommand{\ip}[2]{(a\nabla #1,\nabla #2)}
\newcommand{\tr}{\operatorname{tr}}
\newcommand{\Span}{\operatorname{span}}
\newcommand{\osc}{\operatorname{osc}}
\newcommand{\eps}{\varepsilon}
\newcommand{\dd}{\,\mathrm d}
\title[Conforming enrichment of weak Galerkin spaces]{Conforming Enrichment of Weak Galerkin Spaces for Elliptic Problems with Low-Regularity Solutions}
\author{Chunmei Wang}
\AddToHook{cmd/@setauthors/after}{%
  \begingroup\centering\normalfont\footnotesize
  Department of Mathematics, University of Florida, Gainesville, FL 32611, USA\par
  \textit{Correspondence}\enspace\href{mailto:chunmei.wang@ufl.edu}{\texttt{chunmei.wang@ufl.edu}}\par
  \endgroup
  \vspace{-3pt}}
\date{}
\begin{document}
\begin{abstract}
We enrich a weak Galerkin (WG) solver for applications that retain independent cell and face unknowns. Each independent singular profile adds one global coefficient while preserving the numbering of traces on split faces. A moment-preserving reconstruction couples the profile to polynomial tests through volume gradients. The polynomial stiffness is unchanged at the chosen penalty, and a quotient eigenvalue gives the sharp threshold for positivity on the full space. Once the polynomial system has been solved, a single profile requires one further solve with the same factors and a scalar Schur complement. In two-dimensional tests on four locally refined meshes, the WG numerical flux gives a smaller maximum cell imbalance than the unprocessed conforming gradient flux. The constrained conforming solutions have smaller energy errors with fewer coefficients on these meshes. On five meshes fitted to a vertical material interface, residual selection identifies the physical exponent $3/4$ among six candidates. On a diagonal interface the normal derivative of a square-root profile fails to be square integrable on each corner ray. With this profile, both enriched quadratic methods approach second order in energy.
\end{abstract}
\maketitle

\section{Introduction}
Near a reentrant corner, a polynomial weak Galerkin (WG) approximation on a quasi-uniform mesh can converge at a rate set by the corner singularity. We add the actual cell value and face trace of a singular profile to the WG space and fit one global coefficient for each independent profile. A moment-preserving reconstruction couples the profile gradient to the polynomial tests through cell integrals. At a fixed penalty the polynomial block is exactly the original WG stiffness matrix.

The intended application already stores a field on each polygonal cell and an independently numbered trace on each face. The reconstruction can use a fan inside a polygon without adding polynomial WG unknowns. When a neighboring cell is refined, each half of a split side keeps its own trace. The same partitions also admit conforming discretizations, which give smaller energy errors with fewer coefficients in the polygon-fan and split-face tests. On the four split-face grids of Section~\ref{sec:splitfaces}, the normalized maximum cell imbalance of the unprocessed conforming gradient flux is about $4.15\times10^2$ times that of the WG numerical flux on the coarsest grid and $3.21\times10^4$ times on the finest.

For an existing WG solver, setup includes the reconstruction and the reconstructed load. The sharp full-space threshold requires a quotient eigenvalue calculation, and changing the penalty requires new numerical factors. Once the polynomial system has been solved at the chosen penalty, one independent profile requires its mixed column, one solve with the same factors and a scalar Schur complement.

Wang and Ye \cite{WY} introduced an elliptic discretization in which cell and face values determine a discrete weak gradient. Mu, Wang and Ye \cite{MWY} analyzed the stabilized method on polytopes and tested meshes with hanging nodes. Generalizations of the weak gradient allow other choices of local spaces and stabilization \cite{gWG}. Mu, Wang, Ye and Zhao \cite{MWYZinterface} treated elliptic interface problems with discontinuous coefficients, including tests at low regularity. Ye and Zhang \cite{YZfree} obtained a stabilizer-free formulation on polytopes by raising the degree of the weak gradient. Here we start from the stabilized primal WG form.

Yemm's extended hybrid high-order method (X-HHO) \cite{Yemm} and the extended virtual element method (X-VEM) of Benvenuti et al.\ \cite{BCMS} introduce the enrichment locally. Table~\ref{tab:methods} compares the added unknowns and the profile data each method uses in assembly.

\begin{table}[t]
\centering\small
\caption{Singular enrichment on polygonal meshes.}\label{tab:methods}
\begin{tabular}{@{}p{.152\linewidth}@{\hspace{.023\linewidth}}p{.345\linewidth}@{\hspace{.023\linewidth}}p{.457\linewidth}@{}}
\toprule
Method & Enrichment unknowns & Use of the profile $z$ in assembly\\
\midrule
WG (this work) & \raggedright One global coefficient per independent profile in $H_0^1(\Omega)$. &
\raggedright Cell integrals pair $a_T\nabla z$ with the reconstructed test gradient. The polynomial block is the WG stiffness at the chosen penalty.\tabularnewline
\addlinespace
X-HHO \cite{Yemm} & \raggedright Cell and face spaces gain directions from $\Delta z$ and $\partial_n z$ when these enlarge the polynomial spaces. &
\raggedright The local potential space contains $z$. Assumption~2 (A3) puts $\partial_n z$ in $L^2(F)$ on every face.\tabularnewline
\addlinespace
X-VEM \cite{BCMS} & \raggedright One coefficient per enriched node for a given singular function. &
\raggedright The nodal functions multiply the singular profile. The local projector reproduces linear polynomials and the harmonic enrichment.\tabularnewline
\bottomrule
\end{tabular}
\end{table}

The volume pairing in the WG construction is defined for every profile in $H_0^1(\Omega)$, including profiles whose normal derivatives fail to be square integrable on a face. Other enriched virtual elements also treat crack singularities, including the nonconforming method of Artioli and Mascotto \cite{AM} and the conforming method of Droniou, Manzini and Yemm \cite{DMY}.

A conforming companion transfers discrete tests to the energy space before the load is applied. Veeser and Zanotti analyze this construction for nonconforming methods \cite{VZ}. Applications include hybrid high-order methods \cite{EZ} and the Crouzeix--Raviart and Morley elements \cite{CN}.

Grisvard \cite[Chapter~4]{Grisvard} describes the corner singularities of elliptic boundary value problems on polygons. Fix, Gulati and Wakoff incorporated known singular functions into finite element approximations \cite{FGW}. Strang and Fix discuss the approximation error caused by singularities and the use of singular trial functions \cite[Chapter~8]{SF}. Babu\v{s}ka, Kellogg and Pitk\"aranta derived error estimates on meshes refined toward the corners of a polygon \cite{BKP}. Cai and Kim used a cutoff singular function and a dual extraction formula to recover its coefficient for the Dirichlet Poisson problem \cite{CaiKim}. Melenk and Babu\v{s}ka use a partition of unity to incorporate problem-dependent local spaces \cite{MB}. Strouboulis, Babu\v{s}ka and Copps include corner asymptotics among the local enrichments of the generalized finite element method \cite{SBC}.

Parametrized trial functions can supply candidate profiles \cite{AD,ADsing,Siegel,Barucq,CLC}, and orthogonal least squares selects among candidates in a finite list \cite{OLS}.

The first version of this work \cite[Eq.~(3.1)]{WangPreprint} proposed candidates $z$ obtained by multiplying a neural network output by a function that vanishes on the boundary. Projecting the cell values and face traces gives $Q_hz$. Since each projected pair belongs to the full polynomial WG space $\Vh$, one has $\Vh+\Span\{Q_hz_1,\ldots,Q_hz_m\}=\Vh$. The bilinear form and cell load used there therefore give the polynomial WG Galerkin equations on this space. The present construction supersedes that projected enrichment by using the actual cell value and trace $(z,\tr z)$. In Section~\ref{sec:triangle}, adjoining seven projected candidates to six polynomial coordinates leaves rank six. The first actual function raises the rank to seven, with the same background stiffness and reconstructed load.

\noindent\textit{Main contributions.} The cell and face moments in Lemma~\ref{lem:comp} identify the part $\highform$ of the reconstruction energy that lies above the weak-gradient degree. The correction $\jumpform=\sigma J-\highform$ cancels it exactly, so the polynomial block in \eqref{eq:blocks} equals the original WG stiffness at the same $\sigma$ for every penalty. When $J$ is nonzero, Proposition~\ref{prop:inner} and the sharpness argument following it show that the form on the full space is positive definite exactly when $\sigma>\gamma_h$, where $\gamma_h$ is the largest quotient eigenvalue. The default $\sigma_\star=1+\gamma_h$ enforces the stronger bound $\jumpform\ge J$. On the coarsest triangular refinement mesh, this raises the condition number of the polynomial block about fifteenfold relative to unit penalty in the same nodal basis (Section~\ref{sec:penaltytest}). The mixed columns use only volume gradients of the $H_0^1(\Omega)$ profiles. Proposition~\ref{prop:raw} bounds the effect of replacing the reconstructed load by the usual cell load in terms of the cellwise oscillation of $f$. Best approximation and the one-direction error decrease then follow from Galerkin orthogonality (Theorems~\ref{thm:best} and~\ref{thm:gain}).

The numerical study compares prescribed enrichment with conforming elements and corner grading. Section~\ref{sec:refinement} measures assembly and solve costs on triangles. Section~\ref{sec:penaltytest} studies conditioning, and Section~\ref{sec:squares} tracks the fitted corner coefficient on square cells. Section~\ref{sec:transmission} tests profile selection and a square-root singularity at fitted material interfaces.

\smallskip
\noindent\textit{Notation used throughout the paper.}\par\nopagebreak[4]
{\small\noindent\begin{tabular}{@{}l@{\qquad}l@{}}
$\iota z,\ Q_hz$ & actual pair of cell value and face trace, and its polynomial projection\\
$C_h$ & functions in $H_0^1(\Omega)$ of degree at most $k$ on each original cell\\
$\Pi_T$ & vector $L^2(T)$ projection onto $[\Pol_{k-1}(T)]^d$\\
$\Eh v,\ E_{\min}v$ & explicit moment-preserving reconstruction and minimum-energy reconstruction\\
$\R x,\ u_m^c$ & reconstructed field and computed physical solution $u_m^c=\R U_m$\\
$J,\ \highform$ & jump form and energy above the weak-gradient degree\\
$\sigma,\ \gamma_h,\ \sigma_\star$ & chosen penalty, largest quotient eigenvalue, and default penalty $1+\gamma_h$\\
$\jumpform$ & correction $\sigma J-\highform$ to the reconstructed energy\\
$\N,\ S_m,\ \Wm$ & candidate dictionary, selected span, and full enriched trial space\\
$\mathcal E_m$ & error in the norm induced by $\Ah$\\
$B_{h,k}(w),\ \osc_h(f)$ & mesh-weighted broken $H^{k+1}$ seminorm and cellwise load oscillation\\
$q_m(n),\ g_m(n)$ & orthogonalized candidate and its normalized residual score\\
$\mathbf{c}_m(n),\ s_{\psi,h}$ & mixed column for candidate $n$ and scalar Schur complement for $\psi$\\
$N_{\mathrm{dof}}$ & number of free polynomial coefficients before static condensation
\end{tabular}}

\section{Polynomial weak Galerkin space and a conforming reconstruction}\label{sec:background}
\subsection{Model and mesh}
Let $\Omega\subset\mathbb R^d$, $d=2,3$, be a bounded connected Lipschitz polygon or polyhedron. With $H_0^1(\Omega)$ denoting the functions in $H^1(\Omega)$ of zero trace on $\partial\Omega$, consider
\begin{equation}\label{eq:pde}
 u\in H_0^1(\Omega),\qquad b(u,w):=\ip{u}{w}=(f,w)\quad(w\in H_0^1(\Omega)).
\end{equation}
Here $f\in L^2(\Omega)$, $(\cdot,\cdot)$ is the $L^2$ inner product, and the symmetric diffusion matrix $a(x)$ is constant on each member of a fixed finite polygonal or polyhedral material partition of $\Omega$. Assume $\alpha I\le a(x)\le\beta I$ almost everywhere for fixed $0<\alpha\le\beta<\infty$. We write $\en{w}^2=b(w,w)$. The weak solution exists and is unique.

Let $\Th$ be a conforming mesh with planar faces that resolves the material interfaces. Write $a_T$ for the constant value on $T\in\Th$, and $\mathcal F_h$ for the set of all mesh faces, including the boundary faces. A flat side may consist of several mesh faces. Each interior mesh face is shared in its entirety by its two incident cells. Write $h_T=\operatorname{diam}T$ and $h=\max_T h_T$. Each cell $T$ is star-shaped with respect to a ball of radius at least $c h_T$. Its subdivision into at most $M$ uniformly shape-regular simplices $K$, with $h_K=\operatorname{diam}K\simeq h_T$, matches the subdivisions of neighboring cells on their common faces. Write $\Sh$ for the resulting conforming triangulation. Fix a polynomial degree $k\ge1$. At each degree-$k$ Lagrange node $p$, the star $\omega_p=\{K\in\Sh:p\in\overline K\}$ has uniformly bounded cardinality and comparable diameters. Any two of its simplices are joined within $\omega_p$ by a chain crossing facets containing $p$. For a boundary node such a chain reaches a boundary facet containing $p$. All mesh constants are independent of $h$. Shape-regular simplicial meshes and uniform square grids with a compatible diagonal subdivision are examples.

For $j\ge0$, the space $\Pol_j(D)$ comprises polynomials of total degree at most $j$ on $D$. A pair $v=(v_0,v_b)\in\Vh$ has $v_0|_T\in\Pol_k(T)$ and $v_b|_F\in\Pol_k(F)$ on each mesh face $F\in\mathcal F_h$. The face value is single-valued and vanishes on $\partial\Omega$. We use $\langle\cdot,\cdot\rangle_{\partial T}$ for the sum of face integrals. The weak gradient belongs to $[\Pol_{k-1}(T)]^d$ and satisfies
\begin{equation}\label{eq:wg}
 (\nabla_w v,q)_T=-(v_0,\nabla\!\cdot q)_T+
       \langle v_b,q\cdot\nu_T\rangle_{\partial T}\qquad(q\in[\Pol_{k-1}(T)]^d),
\end{equation}
where $\nu_T$ is the outward normal. For $w\in H^1(T)$ write $Q_Tw=(Q_0w,Q_bw)$ for the $L^2$ projections of $w$ onto $\Pol_k(T)$ and of its trace onto $\Pol_k(F)$ on each face $F\subset\partial T$. With $\Pi_T$ the vector $L^2(T)$ projection onto $[\Pol_{k-1}(T)]^d$, integration by parts gives $\nabla_wQ_Tw=\Pi_T\nabla w$. For $w\in H_0^1(\Omega)$ the local pairs agree on interior faces and have zero boundary face values, so they define $Q_hw\in\Vh$. Set
\begin{align}
 J(v,w)&=\sum_{T\in\Th}h_T^{-1}\langle v_0-v_b,w_0-w_b\rangle_{\partial T},\label{eq:jump}\\
 a_{\mathrm{WG}}^{\sigma}(v,w)&=\sum_{T\in\Th}(a_T\nabla_wv,\nabla_ww)_T+\sigma J(v,w).\label{eq:base}
\end{align}
For $\sigma>0$ this form is positive definite on $\Vh$. Indeed, $J(v,v)=0$ makes $v_0$ conforming with zero boundary trace, so $\nabla_wv=\nabla v_0$. The polynomial WG scheme with the usual load solves $a_{\mathrm{WG}}^\sigma(v_h,w)=(f,w_0)$. Mu, Wang and Ye \cite[Sec.~4.2]{MWY} use degree $k$ on cells and faces in their Eq.~(17) and take $r=k-1$ in their discrete weak gradient, Eq.~(12). For $u\in H^{k+1}(\Omega)$, a fixed positive penalty and their mesh assumptions A1--A4, estimate~(49) of their Theorem~8.1 bounds the discrete energy error relative to $Q_hu$ by $Ch^k\|u\|_{H^{k+1}(\Omega)}$. We derive the cell-load estimate for the mesh assumptions of this paper after Proposition~\ref{prop:raw}. We specify the penalty and the load used for enrichment below.

\subsection{Construction of the reconstruction}
The construction follows the averaging and bubble corrections used by Ern and Zanotti for hybrid high-order methods \cite{EZ}. Here the corrections preserve all degree-$k$ WG moments on the original cells and faces. These moments enter both the stiffness identity and the load comparison.
\begin{lemma}\label{lem:comp}
There is a linear operator $\Eh:\Vh\to H_0^1(\Omega)$, piecewise in $\Pol_{k+d+1}$ on $\Sh$, such that
\begin{align}
 Q_h\Eh v=v,\qquad &\Pi_T\nabla\Eh v=\nabla_wv\quad(T\in\Th),\label{eq:rightinverse}\\
 \sum_{T\in\Th}\big(h_T^{-2}\|\Eh v-v_0\|_T^2+
       \|\nabla(\Eh v-v_0)\|_T^2\big)&\le C_E J(v,v).\label{eq:comp}
\end{align}
Here $\|\cdot\|_T$ is the $L^2(T)$ norm. The constant depends only on $d,k$ and the mesh hypotheses. If the pair is the cell value and trace of a conforming coarse piecewise $\Pol_k$ function, $\Eh v=v_0$.
\end{lemma}
\begin{proof}
We begin with nodal averaging on the simplicial subdivision, as in \cite[Eq.~(6.7)]{EG}. At an interior degree-$k$ Lagrange node $p$, set $Av(p)$ to the arithmetic mean of $v_0|_K(p)$ over $K\in\omega_p$. Set $Av(p)=0$ at boundary nodes. These are the nodal values of a continuous piecewise $\Pol_k$ function with zero trace. To estimate $v_0|_K(p)-Av(p)$ at an interior node, telescope from $K$ to each simplex in the mean along a facet chain through $p$. At a boundary node telescope to a boundary facet, including the last difference from $v_b(p)=0$. Values coincide within a coarse cell. At a coarse interface insert the common value of the face polynomial at $p$ between the two cell values. For each resulting discrepancy on a simplicial facet $e\subset\partial T$, with $h_e=\operatorname{diam}e$, polynomial norm equivalence gives
\[
 |(v_0|_T-v_b)(p)|^2\le C h_e^{1-d}\|v_0|_T-v_b\|_e^2.
\]
Cauchy--Schwarz on the bounded facet chains bounds each squared nodal difference by a constant times the sum of the squared discrepancies. Scaling the Lagrange basis on $K$ gives
\[
 h_K^{-2}\|Av-v_0\|_K^2+\|\nabla(Av-v_0)\|_K^2
 \le C h_K^{d-2}\sum_{p\text{ a node of }K}\big|Av(p)-(v_0|_K)(p)\big|^2.
\]
Summing over $K$ therefore gives
\[
 \sum_{K\in\Sh}\big(h_K^{-2}\|Av-v_0\|_K^2+
              \|\nabla(Av-v_0)\|_K^2\big)\le C J(v,v).
\]
Indeed $h_K^{d-2}h_e^{1-d}\simeq h_T^{-1}$ within the star. Each facet has a fixed number of nodes and each node a bounded number of incident simplices, so each facet discrepancy occurs only a bounded number of times.

For each simplicial facet $e$ in the interior coarse skeleton, let $b_e$ denote the product of the barycentric coordinates of its vertices, extended to its two incident simplices. It is zero on their other facets. Solve
\[
   \int_e b_e p_e q=\int_e(v_b-Av)q\qquad(q\in\Pol_k(e))
\]
for $p_e\in\Pol_k(e)$, and extend $p_e$ on each adjacent simplex by the same affine-reference polynomial extension of its trace. For example, express $p_e$ in $d-1$ independent facet barycentric coordinates and use the corresponding simplex coordinates. Add $b_ep_e$ on the two simplices to $Av$. This preserves continuity. Boundary facets need no correction. Denote the result after all facet additions by $A_1v$. On every $K$ take the full simplex bubble $b_K$ and determine $p_K\in\Pol_k(K)$ from
\[
    \int_K b_K p_K q=\int_K (v_0-A_1v)q\qquad(q\in\Pol_k(K)).
\]
Set $\Eh v=A_1v+b_Kp_K$ on $K$. The weighted mass matrices are positive definite, since the bubbles are positive in the relative interiors. Summing the subfacet and simplex moment identities against the restrictions of each coarse test polynomial gives $Q_h\Eh v=v$. Testing \eqref{eq:wg} and integrating $\Eh v$ by parts proves the gradient identity.

For a simplex $K$ set $\|w\|_{*,K}^2=h_K^{-2}\|w\|_K^2+\|\nabla w\|_K^2$. Its squared norm scales as $h_K^{d-2}\|\widehat w\|_{H^1(\widehat K)}^2$ under the reference map. On the reference facet the weighted mass inverse bounds the correction polynomial by the $L^2$ norm of the discrepancy, and polynomial extension into the reference simplex is bounded at fixed degree. The squared facet and cell $L^2$ norms scale by $h_K^{d-1}$ and $h_K^d$, respectively. Taking square roots gives
\[
 \|b_ep_e\|_{*,K}\le C h_e^{-1/2}\|v_b-Av\|_e \quad(e\subset\partial K),
 \qquad \|b_Kp_K\|_{*,K}\le C h_K^{-1}\|v_0-A_1v\|_K.
\]
For a corrected facet adjacent to $K$, the inverse trace inequality bounds $\|v_b-Av\|_e$ by $\|v_b-v_0\|_e+C h_K^{-1/2}\|v_0-Av\|_K$. Each $K$ receives at most $d+1$ facet corrections, and each such facet has two incident simplices. Squaring and summing therefore gives $\sum_K \|A_1v-v_0\|_{*,K}^2\le C J(v,v)$ by the averaging estimate. The cell correction bound and $h_K\simeq h_T$ now give \eqref{eq:comp}. For a conforming polynomial pair averaging reproduces the function and all correction right-hand sides vanish.
\end{proof}

\section{Conforming enrichment and residual selection}\label{sec:method}
\subsection{An extension of the polynomial stiffness}
Write $\iota z=(z|_T,\tr z|_{\partial T})_T$ for the actual cell values and traces of $z\in H_0^1(\Omega)$. Let $C_h=\{p\in H_0^1(\Omega):p|_T\in\Pol_k(T)\}$, so $\Vh\cap\iota H_0^1(\Omega)=\iota C_h$. Identify a cell field and its single-valued face field with an element of $\big(\prod_{T\in\Th} L^2(T)\big)\times\big(\prod_{F\in\mathcal F_h} L^2(F)\big)$. For $\iota z$, the two traces on an interior face agree almost everywhere. In this product space form the algebraic sum
\[
    \mathcal X_h=\Vh+\iota H_0^1(\Omega).
\]
For $x=v+\iota z$ define its physical reconstruction $\R x=\Eh v+z$. Two representations differ by a conforming polynomial pair, which Lemma~\ref{lem:comp} reproduces. Thus $\R x$ depends only on $x$. In particular, $\R\iota z=z$.

The part of the reconstruction gradient above degree $k-1$ defines a symmetric form
\begin{equation}\label{eq:K}
 \highform(v,w)=\sum_{T\in\Th}\big(a_T(I-\Pi_T)\nabla\Eh v,(I-\Pi_T)\nabla\Eh w\big)_T .
\end{equation}
The constancy of $a_T$ and \eqref{eq:rightinverse} give the orthogonal splitting
\begin{equation}\label{eq:split}
 b(\Eh v,\Eh w)=\sum_T(a_T\nabla_wv,\nabla_ww)_T+\highform(v,w).
\end{equation}
Since $\nabla v_0|_T\in[\Pol_{k-1}(T)]^d$, \eqref{eq:comp} implies $0\le \highform(v,v)\le\beta C_E J(v,v)$. Set
\begin{equation}\label{eq:gamma}
 \gamma_h=\sup_{J(v,v)>0}\frac{\highform(v,v)}{J(v,v)},\quad
 \sigma_\star:=1+\gamma_h,\qquad \jumpform(v,w)=\sigma J(v,w)-\highform(v,w).
\end{equation}
The parameter $\sigma$ denotes the penalty in use. We take $\sigma=\sigma_\star$ unless another value is specified. If $J$ vanishes identically, set $\gamma_h=0$. Otherwise, $\gamma_h$ is the largest generalized eigenvalue on the quotient of the polynomial space by $\iota C_h$. The quotient eigenproblem depends on the mesh, $a$, $k$ and the chosen reconstruction, including the simplicial subdivision. For the lift in Lemma~\ref{lem:comp}, all degree-$k$ extensions of the same facet polynomial give the same $\Eh$. On each adjacent simplex $K$, the difference of two extensions is divisible by the barycentric coordinate opposite the facet $e$. Multiplication by $b_e$ gives $b_K q$ with $q\in\Pol_{k-1}(K)$, which the cell moment correction cancels. The quotient is independent of the load and the enrichment. Any uniformly bounded upper estimate $\bar\gamma\ge\gamma_h$ also gives an admissible choice $\sigma=1+\bar\gamma$. At $\sigma=\sigma_\star$,
\begin{equation}\label{eq:Dbounds}
 J(v,v)\le \jumpform(v,v)\le(1+\beta C_E)J(v,v)
\end{equation}
by the definition of $\gamma_h$ and the reconstruction bound. The term $\jumpform$ vanishes in either argument on $\iota C_h$.

A single unsplit $d$-simplex with homogeneous Dirichlet data illustrates the degree required by moment preservation. A polynomial of degree $m$ vanishing on every face is divisible by the product of the $d+1$ barycentric coordinates. Preserving all cell moments of arbitrary $v_0\in\Pol_k(T)$ therefore requires $m\ge k+d+1$ by dimension count. This explains the degree in Lemma~\ref{lem:comp} on that simplex. It gives no universal value of $\gamma_h$ for other subdivisions or reconstructions.

A patchwise upper bound also gives a uniform penalty. Write $J_R$ and
$\mathcal K_T$ for the cell contributions to $J$ and $\highform$.
Let $\mathcal P_T$ consist of all coarse cells containing a simplex incident to a
degree-$k$ Lagrange node of a simplex in $T$. Set
$m_R=\#\{T:R\in\mathcal P_T\}$ and
$\widehat J_T=\sum_{R\in\mathcal P_T}m_R^{-1}J_R$, so that $\sum_T \widehat J_T=J$.
For a simplex in $T$, every face chain used in the nodal averaging estimate of Lemma~\ref{lem:comp} lies in $\mathcal P_T$. The subsequent bubble corrections on that simplex use only its facet traces and cell moments. The same estimates, summed just over the simplices in $T$, thus give
\[
 \|\nabla(\Eh v-v_0)\|_T^2\le C\sum_{R\in\mathcal P_T}J_R(v,v).
\]
Since $\Pi_T$ reproduces $\nabla v_0$, the left-hand side bounds $\mathcal K_T(v,v)/\beta$.
Because every $J_R$ is nonnegative, the sum on the right is at most $(\max_{R\in\mathcal P_T}m_R)\widehat J_T(v,v)$. Thus $\ker \widehat J_T\subseteq\ker\mathcal K_T$ on the restrictions of
$\Vh$ to the patch. Let $\rho_T$ be the largest eigenvalue of
$(\mathcal K_T,\widehat J_T)$ on the quotient by $\ker \widehat J_T$, with $\rho_T=0$ for a
zero-dimensional quotient. Then
\[
 \highform(v,v)\le\sum_T\rho_T \widehat J_T(v,v)\le\bar\gamma J(v,v),
 \qquad \bar\gamma=\max_T\rho_T\le C\beta\max_R m_R .
\]
The patch relation is symmetric because neighboring simplices share the same nodes on their common subsimplex. Each coarse cell contains a bounded number of nodes, and each node belongs to a bounded number of simplices. This bounds both the patch dimensions and $\max_R m_R$ uniformly. Thus $\sigma=1+\bar\gamma$ uses one bounded-size eigenproblem
per cell and gives $J(v,v)\le\jumpform(v,v)\le(1+\bar\gamma)J(v,v)$.
The numerical experiments use $\sigma_\star$ by default. Sections~\ref{sec:hats}, \ref{sec:penaltytest} and~\ref{sec:graded} specify the patchwise, unit-penalty and common-penalty comparisons, respectively.

For $x=v+\iota z$ and $y=w+\iota\zeta$ define
\begin{equation}\label{eq:A}
      \Ah(x,y)= b(\R x,\R y)+\jumpform(v,w),\qquad F(y)=(f,\R y).
\end{equation}
Using \eqref{eq:split} gives the following assembly formula
\begin{equation}\label{eq:blocks}
 \Ah(x,y)=a_{\mathrm{WG}}^{\sigma}(v,w)+b(z,\Eh w)+b(\Eh v,\zeta)+b(z,\zeta).
\end{equation}
The polynomial block is the primal WG stiffness \eqref{eq:base}. Its load is $(f,\Eh w)$, which equals $(f,w_0)$ if $f|_T\in\Pol_k(T)$. Proposition~\ref{prop:raw} quantifies the effect of the usual cell load for general $f$. Starting from a polynomial matrix assembled at $\sigma_0$, adding $(\sigma-\sigma_0)J$ gives the block in \eqref{eq:blocks}. With the mesh, diffusion, reconstruction and penalty fixed, this block and its factorization can be reused across loads and enrichment directions. Each new mixed column pairs first weak derivatives. A globally supported function can couple to basis functions throughout the mesh and produce a dense column.

\begin{proposition}\label{prop:inner}
For $\sigma>\gamma_h$, the form $\Ah$ is an inner product on $\mathcal X_h$. For every $\sigma$, its restriction to $\iota H_0^1(\Omega)$ agrees with the continuous energy and the exact weak solution satisfies
\begin{equation}\label{eq:cons}
    \Ah(\iota u,y)=F(y)\qquad(y\in\mathcal X_h).
\end{equation}
\end{proposition}
\begin{proof}
Both terms of \eqref{eq:A} are independent of the representation. When $\sigma>\gamma_h$, \eqref{eq:gamma} gives $\jumpform(v,v)\ge(\sigma-\gamma_h)J(v,v)$. Hence $\Ah(x,x)=0$ implies $J(v,v)=0$ and $\R x=0$. The first equality gives $v=\iota p$ for some $p\in C_h$, and the second gives $p+z=0$, so $x=0$. For any $\sigma$, \eqref{eq:A} gives $\Ah(\iota z,\iota\zeta)=b(z,\zeta)$. Testing \eqref{eq:pde} with $\R y\in H_0^1(\Omega)$ gives \eqref{eq:cons}.
\end{proof}
The threshold is sharp on $\mathcal X_h$ whenever $J$ is nonzero. To see this, choose $v\in\Vh$ attaining the supremum in \eqref{eq:gamma} and scale it so that $J(v,v)=1$. Then $x=v-\iota(\Eh v)$ is nonzero, $\R x=0$, and $\Ah(x,x)=\sigma-\gamma_h$. Thus the full form on $\mathcal X_h$ has a null direction at $\sigma=\gamma_h$ and takes a negative value when $\sigma<\gamma_h$. For a single independent profile, Section~\ref{sec:penaltytest} gives a scalar test for positive definiteness of that finite trial space, including at penalties below the full-space threshold.

For nonzero $J$ and a fixed reconstruction, $\sigma_\star=1+\gamma_h$ is the smallest penalty that gives $\jumpform\ge J$. Penalties strictly between $\gamma_h$ and $\sigma_\star$ also give positive definiteness on $\mathcal X_h$. The bound $\jumpform\ge(\sigma-\gamma_h)J$ gives uniform control of the jumps when the gap stays bounded below independently of $h$. At $\sigma=\gamma_h+\tfrac12$, the proof of the cell error bound~\eqref{eq:physical} uses $J\le2\jumpform$ in place of $J\le\jumpform$. For a fixed trial pair $x=v+\iota z$, raising $\sigma$ by $t>0$ adds $tJ(v,v)$ to $\Ah(x,x)$. The physical error also depends on the computed solution. Section~\ref{sec:penaltytest} compares the two penalty margins.

\subsection{Growing the trial space}
Let $\N\subset H_0^1(\Omega)$ be a dictionary of conforming functions, such as prescribed corner profiles or nodal hats. Galerkin assembly uses each function, its first weak derivatives and the load. Given selected functions $n_1,\ldots,n_m$, set
\begin{equation}\label{eq:space}
 S_m=\Span\{n_1,\ldots,n_m\},\quad \Wm=\Vh+\iota S_m,\quad
 U_m\in\Wm:\quad \Ah(U_m,y)=F(y)\quad(y\in\Wm).
\end{equation}
Here $S_0=\{0\}$. Proposition~\ref{prop:inner} gives uniqueness. The physical approximation is $u_m^c=\R U_m$. A candidate adds one dimension precisely when $n\notin C_h+S_m$. Indeed, $\iota n=v+\iota\zeta$ with $\zeta\in S_m$ forces $v=\iota(n-\zeta)\in\iota C_h$. This criterion applies even if $Q_h n=0$. The space of reconstructed fields is $\R\Wm=\Eh(\Vh)+S_m$. Its dimension increases precisely when $n\notin\Eh(\Vh)+S_m$. Section~\ref{sec:test} verifies growth of both spaces for the dictionaries used in the experiments.

An example isolates the effect of using the actual cell value and trace. On a single unit square with $k=1$, set
\[
    n(x,y)=x(1-x)\,y(1-y)\bigl(|2x-1|^{5/2}-12/77\bigr).
\]
The trace vanishes. The integrals of $t(1-t)$ and $t(1-t)|2t-1|^{5/2}$ over $(0,1)$ equal $1/6$ and $2/77$. Since $(2/77)/(1/6)=12/77$, symmetry gives zero moments of $n$ against $1,x,y$. Thus $Q_hn=0$ although $\en n>0$. The right-inverse identity rules out $n=\Eh v$ because it would imply $v=Q_hn=0$. Since $n\in H^2$, we can take $a=I$ and $f=-\Delta n$. Adjoining $\iota n$ then recovers the exact solution by consistency.

Let $P_m$ denote the $\Ah$-orthogonal projection onto $\Wm$, and write $\hn{x}^2=\Ah(x,x)$. For each candidate $n$ define its orthogonalized direction $q_m(n)$, residual $r_m(n)$ and score $g_m(n)$ by
\begin{align}
       q_m(n)&=(I-P_m)\iota n,\label{eq:dir}\\
       r_m(n)&=F(\iota n)-\Ah(U_m,\iota n)=(f,n)-b(u_m^c,n),\label{eq:res}\\
       g_m(n)&=\frac{|r_m(n)|}{\hn{q_m(n)}}\qquad(\hn{q_m(n)}>0).\label{eq:score}
\end{align}
The normalized residual in \eqref{eq:score} is the orthogonal least-squares selection criterion \cite{OLS} in the $\Ah$ inner product. In an independent basis $\{\varphi_i\}$ of $\Wm$, let $A_m=(\Ah(\varphi_i,\varphi_j))$ and define the mixed column $\mathbf{c}_m(n)$ by $[\mathbf{c}_m(n)]_i=b(\R\varphi_i,n)$. The squared denominator is $b(n,n)-\mathbf{c}_m(n)^\top A_m^{-1}\mathbf{c}_m(n)$. Thus a parameterized candidate $n_\theta$ can be trained by maximizing
\[
   \frac{|(f,n_\theta)-b(u_m^c,n_\theta)|^2}
        {b(n_\theta,n_\theta)-\mathbf{c}_m(n_\theta)^\top A_m^{-1}\mathbf{c}_m(n_\theta)}
\]
over parameters with a positive denominator. The denominator can also be evaluated from the orthogonalized function, avoiding subtraction for nearly dependent directions.

\begin{algorithm}[Residual selection on a fixed WG mesh]\label{alg:enrich}
\leavevmode
\begin{enumerate}
\item Assemble $\Eh$, $J$, $\highform$ and a penalty from \eqref{eq:gamma}. Solve the polynomial block with the reconstructed load. Set $m=0$.
\item Stop if the direction budget has been reached. Otherwise search for an independent candidate with a large score \eqref{eq:score}. Scan the remaining candidates in a finite dictionary, or optimize the score over the candidate parameters. Stop if the search returns no score above the prescribed tolerance.
\item For an accepted candidate $n$ with $g_m(n)>0$, update along $q_m(n)$ using
\begin{equation}\label{eq:update}
 W_{m+1}=\Wm+\Span\{q_m(n)\},\qquad U_{m+1}=U_m+\frac{r_m(n)}{\hn{q_m(n)}^2}q_m(n).
\end{equation}
Equivalently, solve \eqref{eq:space} in the enlarged space. Increase $m$ and return to step~2.
\end{enumerate}
\end{algorithm}
After an exhaustive finite search, a score tolerance $\tau$ bounds the best remaining one-step decrease in squared method error by $\tau^2$. A local search yields a score only for the directions it examines. The variational identities below concern exact integration and exact Galerkin solves. In floating-point assembly one can test independence relative to $\hn{\iota n}$ and reorthogonalize the accepted column. For general nonlinear networks, quadrature accuracy and optimization tolerances are additional implementation choices.

\section{Energy error and approximation of singular solutions}\label{sec:analysis}
\subsection{What the discrete problem minimizes}
\begin{theorem}[Best approximation and physical energy error]\label{thm:best}
For the weak solution $u$ of \eqref{eq:pde} and any selected span, the squared method error $\mathcal E_m^2=\hn{\iota u-U_m}^2$ satisfies
\begin{equation}\label{eq:best}
 \mathcal E_m^2=
 \min_{\substack{v\in\Vh\\z\in S_m}}
       \big\{\en{u-\Eh v-z}^2+\jumpform(v,v)\big\}.
\end{equation}
In particular $\en{u-u_m^c}\le\mathcal E_m$. Writing $U_m=v_m+\iota z_m$, the cell approximation also satisfies
\begin{equation}\label{eq:physical}
 \left(\sum_T\|a_T^{1/2}\nabla(u-v_{m,0}-z_m)\|_T^2+
                J(v_m,v_m)\right)^{1/2}\le C\mathcal E_m .
\end{equation}
The constant is independent of the size, parameters and conditioning of the enrichment basis.
\end{theorem}
\begin{proof}
Subtracting \eqref{eq:space} from \eqref{eq:cons} gives $U_m=P_m\iota u$. Orthogonal projection minimizes the squared norm. Expanding it by \eqref{eq:A} gives \eqref{eq:best}. For \eqref{eq:physical}, insert $\Eh v_m$ between $u-z_m$ and $v_{m,0}$ and apply \eqref{eq:comp} and $J\le\jumpform$. Both the reconstructed function and the combined cell value are independent of the decomposition of $U_m$.
\end{proof}

\begin{theorem}[Decrease obtained from one candidate]\label{thm:gain}
If $\hn{q_m(n)}>0$, the update \eqref{eq:update} is the Galerkin solution in the new space and
\begin{equation}\label{eq:gain}
       \mathcal E_{m+1}^2=\mathcal E_m^2-g_m(n)^2.
\end{equation}
For a finite dictionary, exhaustive selection of a nonzero score until none remains gives the projection onto $\Vh+\iota\Span\N$ in at most $\dim(\Vh+\iota\Span\N)-\dim\Vh$ accepted steps. \end{theorem}
\begin{proof}
Let $e=\iota u-U_m$. Galerkin orthogonality gives $r_m(n)=\Ah(e,q_m(n))$. The new space is the orthogonal sum of $\Wm$ and the line through $q_m(n)$. Projecting $e$ onto that line proves the update and \eqref{eq:gain}. In a finite dictionary the dimension increases at each accepted step. At exhaustion the residual annihilates every generator and the initial polynomial space, proving the projection claim.
\end{proof}
A contraction estimate needs a stronger search hypothesis. If the individual candidates are energy-dense in $H_0^1(\Omega)$, take $n_j\to u$. Then $q_m(n_j)\to\iota u-U_m$ and, for $\mathcal E_m>0$, $g_m(n_j)\to\mathcal E_m$, the Cauchy--Schwarz upper bound. A search attaining a fraction $\vartheta\in(0,1)$ of this supremum therefore gives $\mathcal E_{m+1}\le\sqrt{1-\vartheta^2}\,\mathcal E_m$ by \eqref{eq:gain}. The finite-dictionary conclusions require only the hypotheses of Theorem~\ref{thm:gain}.

The monotonicity in \eqref{eq:gain} concerns the method error including $\jumpform$. The physical energy error has its own residual identity, obtained by testing \eqref{eq:pde} in the continuous space,
\[
  \en{u-u_m^c}=\sup_{0\ne\zeta\in H_0^1(\Omega)}
       \frac{|(f,\zeta)-b(u_m^c,\zeta)|}{\en\zeta}.
\]
This component can change nonmonotonically as the jump contribution changes. A small maximizing score for a finite dictionary bounds a one-step gain. Bounds on an unresolved PDE error would require information about directions outside that dictionary.

\subsection{Rates for a regular part and an energy-space remainder}
For piecewise $H^{k+1}$ functions put
\[
    B_{h,k}(w)=\left(\sum_{T\in\Th}h_T^{2k}|w|_{H^{k+1}(T)}^2\right)^{1/2}.
\]
\begin{theorem}[Regular--singular estimate]\label{thm:singular}
Let $u=u_r+u_s$, where both components belong to $H_0^1(\Omega)$ and $u_r|_T\in H^{k+1}(T)$. Put $\eps_m=\inf_{z\in S_m}\en{u_s-z}$. With constants uniform in $h$ and in the enrichment,
\begin{equation}\label{eq:singular}
    \en{u-u_m^c}\le\mathcal E_m\le C B_{h,k}(u_r)+\eps_m .
\end{equation}
For a purely smooth solution the bound is of order $h^k$. If $u\in C_h+S_m$, both errors vanish. No equation or normal trace for either component individually is required.
\end{theorem}
\begin{proof}
For $w=u_r$, the trace inequality and degree-$k$ polynomial approximation on a star-shaped cell give
\[
 \|w-Q_0w\|_T+h_T\|\nabla(w-Q_0w)\|_T+
    h_T^{1/2}\|w-Q_0w\|_{\partial T}
     \le C h_T^{k+1}|w|_{H^{k+1}(T)} .
\]
Indeed, choose a degree-$k$ averaged Taylor polynomial with the stated volume estimates. The $L^2$ best-approximation property and a polynomial inverse inequality transfer those estimates to $Q_0w$. The scaled trace inequality gives the boundary term. All constants are uniform under the mesh hypotheses. On each face $Q_0w-Q_bw=Q_b(Q_0w-w)$, since the trace of $Q_0w$ has degree at most $k$. Thus $J(Q_hw,Q_hw)^{1/2}\le C B_{h,k}(w)$. Applying \eqref{eq:comp} and \eqref{eq:Dbounds} yields
\[
  \en{w-\Eh Q_hw} + \jumpform(Q_hw,Q_hw)^{1/2}\le C B_{h,k}(w).
\]
Use $v=Q_hu_r$ and any $z\in S_m$ in \eqref{eq:best}. The triangle inequality in the product energy bounds the trial error by the regular error just estimated plus $\en{u_s-z}$. Taking the infimum proves \eqref{eq:singular}. Finally $u=p+z$ with $p\in C_h$ gives the exact trial pair $\iota u$.
\end{proof}
\begin{proposition}[Effect of the usual cell load]\label{prop:raw}
Use the penalty $\sigma=\sigma_\star$ defined in \eqref{eq:gamma}. Let $U_m^{\mathrm{cell}}$ solve \eqref{eq:space} on the same space with right-hand side $(f,w_0+\zeta)$ for $y=w+\iota\zeta$. Then
\begin{equation}\label{eq:raw}
 \hn{U_m^{\mathrm{cell}}-U_m}\le C\osc_h(f),\qquad
 \osc_h(f)^2=\sum_T h_T^2\|f-Q_0f\|_T^2 .
\end{equation}
At $m=0$ this cell-load variant is precisely the polynomial WG method \eqref{eq:base}.
\end{proposition}
\begin{proof}
For $y=w+\iota\zeta$ the two loads differ by $(f,w_0-\Eh w)$. Cell moments remove $Q_0f$ from this pairing. Weighted Cauchy--Schwarz and \eqref{eq:comp} give
\[
 |(f,w_0-\Eh w)|\le \osc_h(f)
    \Big(\sum_T h_T^{-2}\|w_0-\Eh w\|_T^2\Big)^{1/2}
 \le C_E^{1/2}\osc_h(f)J(w,w)^{1/2}.
\]
By \eqref{eq:Dbounds} and \eqref{eq:A}, $J(w,w)\le\jumpform(w,w)\le\hn y^2$. Subtract the two Galerkin equations and take $y=U_m^{\mathrm{cell}}-U_m$ to obtain \eqref{eq:raw}.
\end{proof}

For $\sigma=\sigma_\star$ and $u|_T\in H^{k+1}(T)$ on every cell, we obtain the cell-load energy rate on the present meshes. Put $\|w\|_{\rm WG}^2=a_{\rm WG}^{\sigma}(w,w)$ on $\Vh$. For $m=0$, \eqref{eq:blocks} identifies this norm with $\hn w$, and projection gives $\|U_0-Q_hu\|_{\rm WG}\le\hn{\iota u-Q_hu}$. In conjunction with Proposition~\ref{prop:raw} and the comparison estimate in the proof of Theorem~\ref{thm:singular},
\[
   \|U_0^{\rm cell}-Q_hu\|_{\rm WG}
        \le C\bigl(\osc_h(f)+B_{h,k}(u)\bigr)\le C B_{h,k}(u).
\]
Indeed $f=-a_T:\nabla^2u$ almost everywhere on $T$, where the colon denotes matrix contraction. For $k=1$ use $\|f-Q_0f\|_T\le\|f\|_T\le C|u|_{H^2(T)}$. For $k\ge2$, polynomial approximation gives $\|f-Q_0f\|_T\le C h_T^{k-1}|f|_{H^{k-1}(T)}\le C h_T^{k-1}|u|_{H^{k+1}(T)}$.

\subsection{Approximation and dictionary size at a reentrant corner}\label{sec:corner-approx}
The next result gives an existence bound using the full span of an auxiliary hat dictionary. Let $\Omega$ be an L-shaped polygon with opening angle $3\pi/2$ at the origin. Use polar coordinates $0<\theta<3\pi/2$ near the origin and put $\lambda=2/3$. Choose a smooth radial cutoff $\chi$ equal to one near zero and supported inside the corner neighborhood. For $a=I$ set
\begin{equation}\label{eq:corner}
    u(r,\theta)=\chi(r)r^\lambda\sin(\lambda\theta).
\end{equation}
This function lies in $H_0^1(\Omega)$ and is harmonic near the corner. Integration over the circular arc of radius $\varepsilon$ produces a boundary term of order $\varepsilon^\lambda$ against smooth tests. Thus its weak Laplacian is supported where the cutoff changes and $f=-\Delta u\in L^2(\Omega)$.

\begin{proposition}[Full-dictionary existence bound at a reentrant corner]\label{prop:corner}
Use quasi-uniform, self-similar triangular WG meshes. Assume that at each mesh level at least one triangle incident to the corner has the form $h\widehat K$, where $\widehat K$ is fixed, nondegenerate and has a vertex at the origin. Use each WG triangle as its own reconstruction simplex. Let $\N_\delta$ be the interior nodal hats of a shape-regular, quasi-uniform auxiliary triangulation of size $\delta$ resolving the polygon and its corner. They span the continuous piecewise affine functions on the auxiliary mesh with zero boundary values. For sufficiently small $h$, the polynomial background and the result $u_{\mathrm{end}}^c$ of exhaustive nonzero-score selection from this dictionary satisfy
\begin{equation}\label{eq:rates}
   \|\nabla(u-u_0^c)\|\ge c h^\lambda,\qquad
   \|\nabla(u-u_{\mathrm{end}}^c)\|\le C\delta^\lambda .
\end{equation}
Consequently, $\delta^\lambda\le h^k$ yields an $O(h^k)$ enriched upper bound on this background mesh. In two dimensions the full dictionary has $O(\delta^{-2})$ hats.
\end{proposition}
\begin{proof}
Write $s=r^\lambda\sin(\lambda\theta)$. On $h\widehat K$ the cutoff is one and every $\Eh v$ has degree at most $k+3$. The gradients of $\Pol_{k+3}(\widehat K)$ form a closed subspace of $L^2(\widehat K)^2$. Zero distance would imply that $s$ agrees with a polynomial up to a constant on the connected reference triangle. On an interior ray with $\sin(\lambda\theta)\ne0$, the radial derivative of $s$ is unbounded at the vertex, whereas a polynomial has bounded radial derivative. The reference distance is therefore positive. Homogeneity and the substitution $x=h\widehat x$ give
\[
 \inf_{q\in\Pol_{k+3}(h\widehat K)}\|\nabla(s-q)\|_{h\widehat K}
 =h^\lambda\inf_{\widehat q\in\Pol_{k+3}(\widehat K)}
            \|\nabla(s-\widehat q)\|_{\widehat K}.
\]
This proves the first bound and also applies to coarse polynomial cell values.

Let $I_\delta u$ be the nodal piecewise affine interpolant. On the $O(1)$ triangles within distance $C_0\delta$ of the corner, its nodal values are $O(\delta^\lambda)$ and both $u$ and $I_\delta u$ have energy seminorm $O(\delta^\lambda)$. Away from these triangles, $|D^2s|\le C r^{\lambda-2}$ near the corner and $u$ is smooth outside a fixed radius $r_*$. The local interpolation estimate bounds the remaining squared error by
\[
 C\delta^2\left(1+\int_{c_0\delta}^{r_*}r^{2\lambda-3}\dd r\right)
 \le C\delta^{2\lambda}.
\]
The zero boundary values put $I_\delta u$ in the interior hat span. Comparing the terminal projection with $\iota I_\delta u$ in Theorems~\ref{thm:best} and~\ref{thm:gain} gives the upper bound.
\end{proof}
For the $3\pi/2$ corner, taking $\delta=h^{3k/2}$ gives an $O(h^k)$ existence bound using the full auxiliary space, with $O(h^{-3k})$ available hats and $O(h^{-2})$ background unknowns. Write $V_\delta=\Span\N_\delta$. Since $\Vh\cap\iota V_\delta=\iota(C_h\cap V_\delta)$, the dimension bound in Theorem~\ref{thm:gain} limits the number of accepted hats to $\dim V_\delta-\dim(C_h\cap V_\delta)$. On the fixed mesh in Section~\ref{sec:hats} this bound is $161-5=156$, and the search accepts 155 hats. The dimension bound gives no sparse selection rate under refinement.

When the singular shape $\psi\in H_0^1(\Omega)$ is known, Theorem~\ref{thm:singular} gives a different estimate. For $u=u_r+a_s\psi$ with regular $u_r$ and unknown scalar $a_s$, the choice $S_1=\Span\{\psi\}$ gives $\eps_1=0$ and $\mathcal E_1\le C B_{h,k}(u_r)$. For $\psi\notin C_h$ this costs one additional coefficient. Section~\ref{sec:refinement} measures this rate with a nonzero regular part.

\section{Numerical experiments}\label{sec:test}
The analysis allows $d=2,3$. All computations are two-dimensional. A three-dimensional edge singularity can have a varying axial amplitude \cite{YAS}. One scalar multiple of a profile represents that variation exactly only when its axial shape is supplied. A single triangle and a fixed L-shaped grid test residual selection. The refinement studies follow a prescribed corner profile on triangles, squares and polygon fans. They also compare corner grading and local refinement with hanging vertices. A fitted vertical interface tests selection among six exponents. At a diagonal material interface the physical square-root profile has finite volume energy, but $\int_F|\partial_n\psi|^2\dd s$ diverges logarithmically on each corner ray $F$. The profile thus lies outside the face regularity assumption of X-HHO.

In each prescribed-profile comparison, WG and the conforming Lagrange elements use the same enrichment. The conforming elements are linear in the constant-diffusion tests and quadratic at the fitted material interfaces. Common-mesh comparisons measure accuracy on a fixed partition. For comparisons across levels we give the actual numbers of free coefficients, including each fitted amplitude. In the tables, $N_{\mathrm{dof}}$ counts only the free polynomial coefficients before static condensation.

We report the relative method error $e_{\rm WG}=\mathcal E_m/\en u$ and the relative physical energy error $e_{\rm phys}=\en{u-\R U_m}/\en u$. Each zero-direction row uses the same reconstructed load and penalty as its enriched solves. Candidate selection uses the assembled matrix and the load from $f$. The exact gradient enters only in error measurement and independent load checks. Our $2/3$ corner problems have zero Dirichlet data. Yemm's L-shaped X-HHO test has the same exponent and nonzero Dirichlet data on part of the outer boundary \cite{Yemm}. The L-shaped X-VEM convergence test of Benvenuti et al.\ uses mixed Dirichlet--Neumann data, an exponent of $1/3$ and a relative difference of strain energies after projection \cite{BCMS}. The WG and conforming errors compared below use the same boundary data, load and physical energy norm within each test.

We assembled the error studies at two tensor Gauss--Legendre orders. The order pairs were 12 and 18 for the single triangle, 12 and 20 for the fixed hat grid, 20 and 24 for triangular refinement and grading, and 24 and 32 for squares, polygonal fans, material interfaces and split faces. The minimum-energy lift comparison used orders 18 and 26. Assembly used the unrounded penalty on each mesh.

\subsection{Two unresolved lines on one triangle}\label{sec:triangle}
In this quadratic one-triangle test the polynomial space has six free coordinates and each candidate is specified by one knot. We can therefore check directly whether the projected or unprojected candidate increases the rank, and compare the residual search with an independent scan, before applying the method on corner meshes.

Take $T=\{(x,y):0<x<1,\;0<y<1-x\}$, $a=I$, $k=2$, and the mesh consisting of $T$. Write $b_T=xy(1-x-y)$, $\xi_+=\max(\xi,0)$ for $\xi\in\mathbb R$ and $n_t=b_T(x-t)_+^2$. The solution and dictionary are
\[
    u=n_{8/25}-\tfrac85 n_{69/100},\qquad
    \N=\{n_{j/8}:j=1,\ldots,7\},\qquad f=-\Delta u.
\]
The second derivative in $x$ jumps on the two interior lines. The solution is in $H^2(T)$ and is outside $H^3(T)$. Neither location occurs in the dictionary. A nonzero coefficient of $n_t$ produces a jump in the second derivative at $t$ that the other dictionary functions do not have. Since the reconstructed background consists of polynomials on $T$, each accepted knot also enlarges the space of reconstructed fields. For reproducibility, on every open strip the load of $n_t$ is
\[
 -\Delta n_t=-(\Delta b_T)(x-t)_+^2
              -4(\partial_x b_T)(x-t)_+-2 b_T\,\mathbf 1_{x>t}.
\]
There is no distribution concentrated on the lines because the first derivatives are continuous.

All quadratic Lagrange nodes are boundary nodes here. Hence $\Eh(v_0,0)=b_Tp_v$, where $p_v\in\Pol_2(T)$ satisfies $(b_Tp_v,q)_T=(v_0,q)_T$ for every $q\in\Pol_2(T)$. With $h_T=\sqrt{2}$, \eqref{eq:gamma} gives $\sigma\approx44.28$. The integration strips are split at the two prescribed locations and at the trial knots.

In the second run of Table~\ref{tab:strip}, the bias $-t$ moves the line at which the second derivative of $n_t$ jumps, with $t\in[1/8,7/8]$. We sample $g_m$ at 33 equally spaced knots and bracket each sampled interior maximum by its two neighbors. When the derivative of $g_m^2$ changes from positive to negative across the bracket, we locate its zero with Brent's method to absolute tolerance $10^{-11}$. Otherwise we maximize the score on the bracket to an absolute knot tolerance of $10^{-10}$. We choose the best score among the sampled and refined knots, retain previous knots and refit all linear coefficients jointly. We discard a candidate if the ratio of its squared orthogonalized norm to its original squared norm is at most the tolerance of the experiment: $10^{-10}$ for the moving knots, $100$ times machine epsilon (about $2.2\times10^{-14}$) for the seven fixed knots, and $10^{-11}$ for the fixed hats in Section~\ref{sec:hats}. At each of the three steps a separate scan over $4{,}097$ equally spaced knots placed the sampled maximum within one grid spacing of the fitted knot. At quadrature orders 12 and 18 the two scores agreed to relative tolerance $1.2\times10^{-6}$.

\begin{table}[htbp]
\centering\small
\caption{Masked squared-ReLU enrichment on one triangle. The continuous search trains the knot $t$ of each new unit. Its budget is three units.}\label{tab:strip}
\begin{tabular}{@{}r rrr c rrr@{}}\toprule
&\multicolumn{3}{c}{Seven fixed knots}&&\multicolumn{3}{c}{Trainable knots}\\
\cmidrule(lr){2-4}\cmidrule(l){6-8}
$m$ & New $t$ & $e_{\rm WG}$ & $e_{\rm phys}$ && New $t$ & $e_{\rm WG}$ & $e_{\rm phys}$\\\midrule
0 & -- & 0.562132 & 0.375625 &  & -- & 0.562132 & 0.375625\\
1 & $2/8$ & 0.131561 & 0.124655 &  & 0.307663 & 0.069906 & 0.069893\\
2 & $3/8$ & 0.086128 & 0.085536 &  & 0.701207 & 0.025688 & 0.024297\\
3 & $5/8$ & 0.046388 & 0.045933 &  & 0.319179 & 0.006295 & 0.006295\\
4 & $6/8$ & 0.041424 & 0.040871 &  & -- & -- & --\\
5 & $1/8$ & 0.036439 & 0.036350 &  & -- & -- & --\\
6 & $4/8$ & 0.035888 & 0.035805 &  & -- & -- & --\\
7 & $7/8$ & 0.035782 & 0.035698 &  & -- & -- & --\\
\bottomrule\end{tabular}
\end{table}

To isolate the choice of representation, we also project all seven fixed candidates to their cell and face moments, retaining the same background stiffness and reconstructed load as in Table~\ref{tab:strip}. Their boundary traces vanish. A rank-revealing solve with all seven projected columns has rank six, the polynomial dimension, and gives the zero-direction error $e_{\rm phys}=0.375625$. Using the actual function at $t=2/8$ adds a seventh independent coordinate and reduces this error to $0.124655$.

The first knot, $0.307663$, stays fixed as later knots are added. After the second knot is placed near the jump at $69/100$, the third, $0.319179$, gives an independent correction closer to the first jump at $8/25$. When the third unit is added and all linear coefficients are refitted, $e_{\rm WG}$ falls from $0.025688$ to $0.006295$.

With the same penalty and the usual cell load $(f,v_0)$, the reconstructed polynomial WG solution has $e_{\mathrm{phys}}=0.373831$. Section~\ref{sec:penaltytest} also reports the unit-penalty result. As a separate exact-representation check, replacing the two prescribed locations by $3/8$ and $5/8$ selects those two knots and gives $e_{\mathrm{phys}}<10^{-13}$ after the second update. In the two finite-dictionary runs the search stops after exhausting the dictionary or when the largest remaining score is at most $10^{-12}$ times its initial maximum.

\subsection{A finite hat dictionary on a 24-cell L-shaped mesh}\label{sec:hats}
This fixed-mesh experiment tests the two conclusions of Theorem~\ref{thm:gain}, the one-step decrease in squared method error and the projection obtained when the residual vanishes on the remaining dictionary.
Let $\Omega=(-1,1)^2\setminus\big([0,1]\times[-1,0]\big)$ and $a=I$. In this test the smooth factor is $\mu(x,y)=(1-x^2)(1-y^2)$. For angles measured counterclockwise from the positive $x$ axis set
\[
    s=r^{2/3}\sin(2\theta/3),\qquad u=\mu s,\qquad
    f=-s\Delta\mu-2\nabla\mu\cdot\nabla s .
\]
The radial cutoffs from Proposition~\ref{prop:corner} are unnecessary for this manufactured solution. The outer sides have $\mu=0$, the two corner rays have $s=0$, and $\mu(0,0)=1$. Thus the corner singularity is present and the boundary data are homogeneous.

Take $k=1$. Divide the domain into squares of side $1/2$ and split each along the southwest--northeast diagonal. This gives 24 coarse triangles with 128 WG unknowns after imposing the boundary condition. Use these triangles as their own reconstruction subdivision. In the facet correction of Lemma~\ref{lem:comp} extend the facet polynomial $p_e$ from its endpoints $i,j$ as $p_e(i)\ell_i+p_e(j)\ell_j$ in the adjacent triangle, where $\ell_i,\ell_j$ are its barycentric coordinates. Formula~\eqref{eq:gamma} gives $\sigma\approx71.92$. The patchwise bound in Section~\ref{sec:method} gives a penalty of approximately $454.36$ on this mesh. The dictionary contains the 161 interior hats on the similarly oriented grid of side $h_f=1/8$. For a center $(x_j,y_j)$ put $X=(x-x_j)/h_f$, $Y=(y-y_j)/h_f$ and take
\[
      n_j(x,y)=\max\{0,1-\max(|X|,|Y|,|X-Y|)\}.
\]
The computation uses these explicit nodal hats. Their representation by rectified linear units is described in \cite{He}. Each accepted hat in Table~\ref{tab:corner} adds one coefficient, and the Galerkin update also adjusts the background coefficients.

We integrate over the 384 fine triangles using a Duffy transformation. With the corner as the first vertex, the additional radial substitution $\rho=\xi^3$ resolves the fractional powers. The reconstruction and hat stiffness integrands are piecewise polynomial.

\begin{table}[htbp]
\centering\small
\caption{Reentrant corner. The same 24 WG cells are used in every row.}\label{tab:corner}
\begin{tabular}{@{}rrrr@{}}\toprule
Added directions $m$ & Total dimension & $e_{\mathrm{WG}}$ & $e_{\mathrm{phys}}$\\\midrule
0&128&0.497305&0.451058\\
8&136&0.437108&0.397843\\
16&144&0.402645&0.363758\\
32&160&0.354417&0.320749\\
64&192&0.283903&0.261327\\
128&256&0.158485&0.154553\\\bottomrule
\end{tabular}
\end{table}

With the cell load, the polynomial WG solution on the same mesh and with the same reconstruction gives $e_{\mathrm{phys}}=0.451394$. The first 32 selected hats reduce the reconstructed-load value from $0.451058$ to $0.320749$. The dictionary and its integration mesh supply additional spatial resolution.

To locate the improvement, let $r$ denote distance to the reentrant vertex. On the background solve we integrated $|\nabla(u-\R U_0)|^2$ over the sector disks using polar quadrature split at coarse triangle edges. The fractions of the total squared physical error in $r\le1/4$ and $r\le1/2$ are $7.69\%$ and $18.40\%$, respectively. Among the first 32 selected centers the smallest radius is $\sqrt{10}/8=0.395285$. The first center in $r\le 1/4$ is $(-1/8,1/8)$ at step 51 (Figure~\ref{fig:centers}). Each center marks the peak of a hat with extended support. The coarse experiment shows a reduction of global error with most of the initial error outside the two disks. It does not isolate the contribution of the smooth multiplier $\mu$ or determine an asymptotic corner rate.

\begin{figure}[htbp]
\centering
\begin{tikzpicture}[
    x=2.65cm,y=2.65cm,font=\small,
    early hat/.style={fill=blue!65!black},
    later hat/.style={draw=black!80,fill=white,line width=.55pt},
    last hat/.style={draw=red!70!black,line width=1.05pt,line cap=round},
    disk radius/.style={draw=black!85,line width=.8pt,dash pattern=on 2.4pt off 1.35pt}
  ]
  \draw[black!30,thin] (-0.5,-1)--(-0.5,1) (0,0)--(0,1) (0.5,0)--(0.5,1) (-1,-0.5)--(0,-0.5) (-1,0)--(0,0) (-1,0.5)--(1,0.5);
  \foreach \a/\b in {-1/-1,-0.5/-1,-1/-0.5,-0.5/-0.5,-1/0,-0.5/0,0/0,0.5/0,-1/0.5,-0.5/0.5,0/0.5,0.5/0.5} \draw[black!30,thin] (\a,\b)--++(0.5,0.5);
  \draw[black!65] (-1,-1)--(0,-1)--(0,0)--(1,0)--(1,1)--(-1,1)--cycle;
  \draw[disk radius] plot[domain=0:270,samples=97] ({0.25*cos(\x)},{0.25*sin(\x)});
  \foreach \xx/\yy in {-0.875/0.625,-0.75/0.75,-0.625/0.875,-0.5/0.5,-0.75/0.625,-0.625/0.75,-0.875/0.75,-0.75/0.875,-0.625/0.625,-0.875/0.875,-0.5/0,0/0.5,-0.5/-0.125,-0.5/-0.25,-0.5/-0.375,-0.5/-0.5,0.125/0.5,0.25/0.5,0.375/0.5,0.5/0.5,-0.5/0.125,-0.375/0.125,-0.125/0.5,-0.125/0.375,-0.5/0.375,-0.5/0.25,-0.375/0.5,-0.25/0.5,-0.625/-0.125,0.125/0.625,-0.375/-0.875,-0.25/-0.75} \fill[early hat] (\xx,\yy) circle[radius=1.9pt];
  \foreach \xx/\yy in {-0.125/-0.625,0.875/0.375,0.75/0.25,0.625/0.125,-0.625/0.5,-0.5/0.625,-0.625/0.375,-0.375/0.625,-0.375/0.375,-0.375/0.25,-0.25/0.375,-0.25/0.25,-0.375/-0.5,0.5/0.375,-0.875/-0.875,0.875/0.875,-0.25/-0.125,0.125/0.25} \draw[later hat] (\xx,\yy) ++(-1.65pt,-1.65pt) rectangle ++(3.3pt,3.3pt);
  \draw[last hat] (-.125,.125) ++(-2.6pt,-2.6pt) -- ++(5.2pt,5.2pt);
  \draw[last hat] (-.125,.125) ++(-2.6pt,2.6pt) -- ++(5.2pt,-5.2pt);
  \node[below left] at (-1,-1) {$(-1,-1)$};
  \node[above right] at (1,1) {$(1,1)$};
  \draw[black!70,line width=.45pt] (0,0)--(.18,-.12);
  \node[anchor=north west,inner sep=1pt] at (.18,-.12) {$(0,0)$};
  \fill[early hat] (1.38,.63) circle[radius=1.9pt]; \node[right] at (1.49,.63) {Steps 1--32};
  \draw[later hat] (1.38,.31) ++(-1.65pt,-1.65pt) rectangle ++(3.3pt,3.3pt); \node[right] at (1.49,.31) {Steps 33--50};
  \draw[last hat] (1.38,-.01) ++(-2.6pt,-2.6pt) -- ++(5.2pt,5.2pt);
  \draw[last hat] (1.38,-.01) ++(-2.6pt,2.6pt) -- ++(5.2pt,-5.2pt); \node[right] at (1.49,-.01) {Step 51};
  \draw[disk radius] (1.25,-.33)--(1.50,-.33); \node[right] at (1.51,-.33) {$r=1/4$};
\end{tikzpicture}
\space
\caption{Centers of the first 51 selected hats on the L-shaped domain. Thin solid lines show the 24 background triangles formed by splitting squares of side $1/2$. The dashed arc marks $r=1/4$ about the origin. The cross at $(-1/8,1/8)$ marks step 51.}\label{fig:centers}
\end{figure}

For a comparison at similar total dimensions, we also solve the conforming piecewise linear ($P_1$) Galerkin problem with the same $u$ and volume load. On each uniform square grid we use the same diagonal orientation and count interior vertex values after imposing the homogeneous boundary condition. We assemble its stiffness independently and integrate the squared gradient error directly over its triangles. For both pairs in Table~\ref{tab:cost}, conforming $P_1$ gives the smaller physical energy error. The dimension counts include all active coefficients before elimination.

\begin{table}[htbp]
\centering\small
\caption{Selected hats on the L-shaped WG mesh and unenriched conforming $P_1$ on finer meshes of comparable dimension. The first $P_1$ grid also supplies the hat dictionary.}\label{tab:cost}
\begin{tabular}{@{}rrrcrrr@{}}\toprule
\multicolumn{3}{c}{Enriched WG, 24 background triangles}&&\multicolumn{3}{c}{Conforming $P_1$}\\
\cmidrule(lr){1-3}\cmidrule(lr){5-7}
Added hats & Dimension & $e_{\mathrm{phys}}$ && Triangles & Dimension & $e_{\mathrm{phys}}$\\\midrule
32  & 160 & 0.320749 && 384 & 161 & 0.165218\\
128 & 256 & 0.154553 && 600 & 261 & 0.135950\\\bottomrule
\end{tabular}
\end{table}

The patchwise penalty ($454.36$) gives larger physical errors than the quotient penalty ($71.92$) on this mesh. The two searches select the same set of 128 hats. On that space the error increases from $0.154553$ to $0.167903$ when the larger penalty is used. On the full space it increases from $0.149840$ to $0.162494$. The patch estimate replaces the global eigenproblem by smaller local problems. For fixed $E_h$, every certified upper bound for $\gamma_h$ yields a penalty at least as large as the global $1+\gamma_h$ when imposing $\jumpform\ge J$.

There is also a comparison on exactly the same auxiliary grid. Let $V_f$ be its conforming $P_1$ space with zero boundary values. Since each reconstructed background function is a polynomial on each coarse triangle, a function in both spaces agrees with an affine function on an open fine triangle inside each coarse triangle. The coarse polynomial is therefore affine throughout that coarse triangle. Lemma~\ref{lem:comp} reproduces every function in $C_h$, so $\Eh(\Vh)\cap V_f=C_h$. Thus each independent hat added in Table~\ref{tab:corner} enlarges the reconstructed space as well as the space of weak pairs. The five coarse interior vertices give $\dim C_h=5$. The full space of weak pairs has dimension $128+161-5=284$. An independent basis consists of the WG unknowns and the fine interior hats whose centers are not coarse vertices. A direct solve on this basis gives $e_{\mathrm{phys}}=0.149840$, compared with $0.165218$ using the 161 fine-grid hats alone. Theorem~\ref{thm:best} predicts this ordering by admitting every $(v,z)=(0,z)$, $z\in V_f$, in the full-space minimization.

Theorem~\ref{thm:gain} allows at most $284-128=156$ accepted directions. A hat $n$ is eligible if $\|q_m(n)\|_{\Ah}^2>10^{-11}b(n,n)$. The search stops when the largest eligible score is at most $10^{-12}$ times its initial maximum. After 155 acceptances two hats remained eligible, and the score ratio was below $4\times10^{-13}$ at both quadrature orders (12 and 20). Reflection across $y=-x$ exchanges the two eligible centers $(-6,3)/8$ and $(-3,6)/8$. The five coarse hats agree at the reflected pair, and their restriction to the six omitted vertices has exact rank five. Equality of the two fine-component values thus characterizes the terminal space independently of the representation. Its missing direction is odd, while the form and load are invariant under reflection, so the exact residual in that direction vanishes. The final 283-dimensional solve agreed with the full projection to relative method-norm error below $10^{-13}$.

\subsection{One localized corner mode under mesh refinement}\label{sec:refinement}
On the same L-shaped domain consider a manufactured solution with a prescribed leading singular function and a nonzero smooth remainder. With $s=r^{2/3}\sin(2\theta/3)$ as in Section~\ref{sec:hats}, put
\[
 \eta(t)=(1-4t^2)_+^3,\qquad \chi(x,y)=\eta(x)\eta(y),\qquad
 \psi=\chi s,\qquad u_r=xy(1-x^2)(1-y^2),\qquad u=\psi+u_r .
\]
Both components have zero boundary trace. The $C^2$ cutoff is supported in $[-1/2,1/2]^2$ and satisfies $\chi=1+O(r^2)$ at the origin. For $a=I$ the volume load used in assembly is
\[
 f=-(\Delta\chi)s-2\nabla\chi\cdot\nabla s+6xy(2-x^2-y^2).
\]
The last term equals $-\Delta u_r$. Harmonicity of $s$ and the vanishing small-circle flux against smooth tests justify this formula at the corner. The added direction is the prescribed mode $\psi$, with known exponent and cutoff. Denote its Galerkin coefficient by $c_h$, so $U_1=w_h+\iota(c_h\psi)$ with $w_h\in\Vh$. The enlarged linear system determines $c_h$ from $f$. On a corner triangle the gradient of $\psi$ has the $r^{-1/3}$ singularity and hence it enlarges both the weak pair space and the reconstructed field space.

We split squares of side $\ell=1/L$, $L=4,8,16,32,64$, along the same diagonals, giving $h=\sqrt{2}\,\ell$. The cutoff lines $|x|=1/2$ and $|y|=1/2$ are mesh lines. Each polynomial block uses $k=1$ and the moment reconstruction of Section~\ref{sec:hats}. The largest quotient eigenvalue gives $\sigma$ between $72.24$ and $72.25$. Both solves use the full WG background, and the enriched solve adds one coefficient. Define
\[
 e_{j,\mathrm{loc}}=\frac{\|\nabla(u-\R U_j)\|_{L^2(\Omega^c_\ell)}}{\|\nabla u\|_{L^2(\Omega)}},
 \qquad \Omega^c_\ell=\bigcup\{T\in\Th:0\in\overline T\},\quad j=0,1.
\]
Here $\Omega^c_\ell$ denotes the corner patch of triangles incident to the origin. Each rate is the base-two logarithm of the ratio of successive global errors, from coarse to fine. The second error panel of Table~\ref{tab:local} uses conforming linear elements on the same triangles, with the same volume load and boundary values. In its enriched solve, the nodal values and the coefficient of the identical $\psi$ are determined jointly.

\begin{table}[htbp]
\centering\small
\caption{Localized corner mode on common triangular meshes. In the two error panels $N_{\mathrm{dof}}$ counts the free background coefficients before static condensation. The enriched solve adds the coefficient $c_h$ of $\psi$. Global and corner-patch errors are normalized by the global energy norm of $u$. The last panel gives the quotient dimension and the median eigensolve time over three single-CPU runs with one BLAS thread.}\label{tab:local}
\begin{tabular}{@{}r r rr rr rr r@{}}\toprule
&&\multicolumn{2}{c}{WG, global}&\multicolumn{2}{c}{WG + mode, global}&\multicolumn{2}{c}{Corner patch}&\\
\cmidrule(lr){3-4}\cmidrule(lr){5-6}\cmidrule(l){7-8}
$\ell$ & $N_{\mathrm{dof}}$ & $e_{\rm phys}$ & Rate & $e_{\rm phys}$ & Rate & $e_{0,\mathrm{loc}}$ & $e_{1,\mathrm{loc}}$ & $c_h$\\\midrule
$1/4$ & 544 & 0.439486 & -- & 0.258121 & -- & 0.276868 & 0.042813 & 0.950216\\
$1/8$ & 2,240 & 0.233565 & 0.912 & 0.132849 & 0.958 & 0.118373 & 0.012376 & 0.951498\\
$1/16$ & 9,088 & 0.125978 & 0.891 & 0.066937 & 0.989 & 0.060936 & 0.003789 & 0.959463\\
$1/32$ & 36,608 & 0.069207 & 0.864 & 0.033542 & 0.997 & 0.036423 & 0.001356 & 0.967929\\
$1/64$ & 146,944 & 0.038968 & 0.829 & 0.016784 & 0.999 & 0.022663 & 0.000564 & 0.975903\\
\midrule
&&\multicolumn{2}{c}{$P_1$, global}&\multicolumn{2}{c}{$P_1+\psi$, global}&\multicolumn{2}{c}{Corner patch}&\\
\cmidrule(lr){3-4}\cmidrule(lr){5-6}\cmidrule(lr){7-8}
$\ell$ & $N_{\mathrm{dof}}$ & $e_{\rm phys}$ & Rate & $e_{\rm phys}$ & Rate & $e_{0,\mathrm{loc}}$ & $e_{1,\mathrm{loc}}$ & $c_h$\\\midrule
$1/4$ & 33 & 0.555989 & -- & 0.323609 & -- & 0.364615 & 0.059539 & 0.955617\\
$1/8$ & 161 & 0.300162 & 0.889 & 0.167208 & 0.953 & 0.163723 & 0.017187 & 0.956541\\
$1/16$ & 705 & 0.164518 & 0.867 & 0.084353 & 0.987 & 0.088074 & 0.005163 & 0.964244\\
$1/32$ & 2,945 & 0.091757 & 0.842 & 0.042285 & 0.996 & 0.053268 & 0.001785 & 0.972328\\
$1/64$ & 12,033 & 0.052419 & 0.808 & 0.021163 & 0.999 & 0.033227 & 0.000713 & 0.979669\\
\bottomrule\end{tabular}
\par\smallskip
\begin{tabular}{@{}lrrrrr@{}}\toprule
$L=1/\ell$ & 4 & 8 & 16 & 32 & 64\\
Quotient dimension & 511 & 2,079 & 8,383 & 33,663 & 134,911\\
Quotient eigensolve (s) & 0.00394 & 0.00976 & 0.0368 & 0.195 & 1.26\\
\bottomrule\end{tabular}
\end{table}

Here $\eps_1=0$ in Theorem~\ref{thm:singular}, with $u_s=\psi$ and the polynomial remainder $u_r$ above, so the enriched global energy error is $O(h)$. For the background, the scaling argument of Proposition~\ref{prop:corner} still applies. Since $\nabla u_r(0)=0$ and $u_r$ has bounded second derivatives, $|\nabla u_r|\le Ch$ on a corner triangle of diameter $h$. Its area is $O(h^2)$, so $\|\nabla u_r\|_T=O(h^2)$. Together with $\chi-1=O(r^2)$ and $\nabla\chi=O(r)$, this gives $\|\nabla[(\chi-1)s+u_r]\|_T=O(h^{2+2/3}+h^2)=o(h^{2/3})$. On each triangular mesh, the conforming nodal interpolant $I_hu$ belongs to $C_h$, so the pair $\iota I_hu$ is an admissible polynomial WG comparison with zero jump. The near-corner and outer-region interpolation estimates in Proposition~\ref{prop:corner}, applied to $\psi$ with $\delta=h$, give $\|\nabla(\psi-I_h\psi)\|\le Ch^{2/3}$. Interpolating $u_r$ adds $O(h)$. Theorem~\ref{thm:best} and the corner lower bound above therefore give both bounds $c h^{2/3}\le\|\nabla(u-\R U_0)\|\le C h^{2/3}$ for sufficiently small $h$.

The measured global rates decrease from $0.912$ to $0.829$ over the reported refinements. On the finest pair the patch rate is $0.685$, close to the corner exponent $2/3$, and the enriched global rate is $0.999$. For $P_1+\psi$ the last global rate is also $0.999$. On the finest mesh the relative errors are $0.016784$ with $146{,}945$ free coefficients for WG with the mode and $0.021163$ with $12{,}034$ for $P_1+\psi$. On all five common meshes the measured physical energy error is smaller for WG with the mode. Theorem~\ref{thm:best} guarantees the non-strict inequality because every enriched conforming linear field is an admissible WG comparison with zero jump. Table~\ref{tab:cost} uses the other manufactured solution and the fixed hat dictionary. Section~\ref{sec:graded} compares grading and enrichment at matched WG background dimensions for the present problem.

For the uniform triangular grids we remove one cell vertex coordinate per interior mesh vertex to represent the quotient by $\iota C_h$. A symmetric generalized eigensolve with relative tolerance $10^{-10}$ gives the largest quotient eigenvalue. The polynomial PDE solve retains all unconstrained WG unknowns and uses sparse direct factorization. The mixed-column solve uses the same factors. At each level we timed uncondensed WG, statically condensed WG and conforming $P_1$, each enriched with the same profile, on one logical CPU of an AMD EPYC 9V74. The last panel of Table~\ref{tab:local} times the quotient eigensolve in the condensed WG run, from assembled quotient matrices to the value of $\sigma_\star$. Every run began with mesh construction and ended with recovery of the solution coefficients. Both WG variants assembled the same reconstruction, load and mixed column, computed the same penalty and used a fresh factorization. At $L=64$ condensation removed 73,728 cell coefficients, leaving 73,216 face coefficients and one amplitude. Relative to $P_1+\psi$, the WG system therefore has $12.21$ times as many unknowns before condensation and $6.08$ times as many afterward. The condensed WG and conforming bordered matrices had 767,429 and 65,930 numerical nonzeros, respectively, counting both halves at a relative entry tolerance of $10^{-12}$.

On the finest grid with order-12 quadrature, the median total times were $23.87$~s for uncondensed WG, $12.00$~s for condensed WG and $0.638$~s for $P_1+\psi$. For condensed WG the median assembly and quotient times were $10.30$~s and $1.26$~s, respectively. Each method had one warm-up followed by three timed runs in rotating order. Error integration and residual checks followed the timer. The two WG polynomial coefficient vectors agreed to relative Euclidean error below $10^{-11}$ on all five meshes.

The coefficient of $\psi$ in the exact decomposition $u=u_r+\psi$ is one. The triangular WG panel of Table~\ref{tab:local} gives $c_h=0.950216$ on the coarsest mesh and $0.975903$ on the finest. To estimate the coefficient error, put $q_\psi=q_0(\psi)=(I-P_0)\iota\psi$. Projection onto the added line in \eqref{eq:update} gives
\[
 c_h-1=\frac{\Ah((I-P_0)\iota u_r,q_\psi)}{\hn{q_\psi}^2}.
\]
Write $w_r=P_0\iota u_r\in\Vh$ and $s_{\psi,h}=\hn{q_\psi}^2$. Galerkin orthogonality lets us replace $q_\psi$ in the numerator by $\iota\psi$. Equation~\eqref{eq:blocks} then gives
\begin{equation}\label{eq:coefficient-pairing}
 s_{\psi,h}(c_h-1)=b(u_r-\Eh w_r,\psi).
\end{equation}
The smooth-part estimate gives $\hn{(I-P_0)\iota u_r}\le Ch$. The corner-triangle lower bound in the proof of Proposition~\ref{prop:corner} applies to $\psi$ because $\|\nabla((\chi-1)s)\|_T=O(h^{2+2/3})$, giving $\hn{q_\psi}\ge c_0 h^{2/3}$ with $c_0>0$. Hence $|1-c_h|\le Ch^{1/3}$. The last two observed triangular WG orders for $|1-c_h|$, $0.338$ and $0.412$, are consistent with this upper bound. Its sharpness remains open. Here $c_h$ is the amplitude in the Galerkin minimizer. Using a separate extraction formula would introduce a postprocessing step with its own error analysis. Blum and Dobrowolski used functions from the kernel of the adjoint problem as tests to recover singular coefficients \cite{BD}. For fixed cutoff radii and the present opening angle $3\pi/2$, the exponent in Cai and Kim's coefficient error bound for their dual extraction formula is $1+\pi/(3\pi/2)=5/3$ \cite[Theorem~4.3, Eq.~(4.15)]{CaiKim}.

\subsection{Penalty choice, conditioning and reuse of the WG factorization}\label{sec:penaltytest}
For a fixed mesh, diffusion, degree and reconstruction, the same quotient threshold can be used for different loads and enrichment profiles. Suppose an existing direct solver has factored the polynomial matrix $A_{\sigma_0}$. At a different penalty $\sigma$ the polynomial block is $A_\sigma=A_{\sigma_0}+(\sigma-\sigma_0)J$, where $J$ denotes the jump matrix. The symbolic sparsity pattern is the same, and the numerical factors must be updated. At fixed $\sigma$ only the right-hand side changes when the cell load is replaced by $(f,E_hw)$. Given the background solution at that penalty and load, one further solve with the mixed column supplies the scalar Schur update.

For $\sigma>0$ the polynomial block $A_\sigma$ is positive definite. For an independent profile $\psi\in H_0^1(\Omega)\setminus C_h$, let $\mathbf c=\mathbf{c}_0(\psi)$ be its mixed column and define
\[
 s_{\psi,h}(\sigma)=b(\psi,\psi)-\mathbf c^\top A_\sigma^{-1}\mathbf c .
\]
Block elimination makes the one-profile matrix congruent to $\operatorname{diag}(A_\sigma,s_{\psi,h}(\sigma))$. The form on $W_1=\Vh+\Span\{\iota\psi\}$ is therefore positive definite exactly when $s_{\psi,h}(\sigma)>0$. This test concerns the specified profile. Proposition~\ref{prop:inner} gives the full-space guarantee, and the bound $\jumpform\ge J$ used in the error estimates holds at $\sigma_\star$.

Table~\ref{tab:penalty} compares the usual polynomial WG solution $v_h^\sigma$ with cell load $(f,w_0)$ for a test pair $w=(w_0,w_b)$ at $\sigma=1$ and at $\sigma_\star=1+\gamma_h$. For each solution we measure its discrete weak gradient directly and also reconstruct the same coefficients with $E_h$. For $a=I$ the two relative errors are
\[
 e_{\rm weak}=\frac{\|\nabla u-\nabla_w v_h^\sigma\|}{\|\nabla u\|},
 \qquad e_{\rm rec}=\frac{\|\nabla(u-E_h v_h^\sigma)\|}{\|\nabla u\|}.
\]

\begin{table}[htbp]
\centering\small
\caption{Polynomial WG. The upper panel uses the usual cell load and evaluates both gradients of each solution. The lower panel compares two reconstructions of degree-one WG on the L-shaped problem of Section~\ref{sec:hats}. Each uses its own reconstructed load and penalty. $\kappa_2(A)$ is the condition number in the original cell and edge basis.}\label{tab:penalty}
\begin{tabular}{@{}lrrr@{}}\toprule
Background & $\sigma$ & $e_{\rm weak}$ & $e_{\rm rec}$\\\midrule
Single triangle ($k=2$) & 1 & 1.039651 & 24.787531\\
                        & 44.28 & 0.967523 & 0.373831\\
24 triangles on the L-shaped domain ($k=1$) & 1 & 0.506848 & 3.391236\\
                        & 71.92 & 0.507704 & 0.451394\\\bottomrule
\end{tabular}
\par\medskip
\begin{tabular}{@{}r rrr rrr@{}}\toprule
 &\multicolumn{3}{c}{Explicit local reconstruction}&\multicolumn{3}{c}{Minimum-energy reconstruction}\\
\cmidrule(lr){2-4}\cmidrule(l){5-7}
Triangles & $\sigma_\star$ & $\kappa_2(A)$ & $e_{\rm phys}$ & $\sigma_\star$ & $\kappa_2(A)$ & $e_{\rm phys}$\\\midrule
6 & 66.46 & 59.26 & 0.787884 & 61.88 & 58.48 & 0.774831\\
24 & 71.92 & 679.18 & 0.451058 & 64.22 & 610.40 & 0.436509\\
96 & 72.24 & 2,504.99 & 0.239960 & 64.43 & 2,240.88 & 0.232112\\
\bottomrule\end{tabular}
\end{table}

The two gradients in the upper panel of Table~\ref{tab:penalty} measure different parts of the polynomial solution. On the single triangle all face values vanish and $\nabla\!\cdot q$ is constant for each $q\in[\Pol_1(T)]^2$. Formula~\eqref{eq:wg} makes the weak gradient depend only on the cell mean. The other cell moments enter $E_hv_h^\sigma$ and are controlled by the penalty. The large reconstructed error at $\sigma=1$ includes components that the weak gradient cannot detect on this mesh. On the 24-triangle mesh the two penalties give nearly equal weak-gradient errors and substantially different reconstructed errors.

The quotient $\highform(v,v)/J(v,v)$ measures reconstruction gradient energy above degree $k-1$ per unit of jump. Both the lift and the geometry enter this ratio. Keeping the single right triangle fixed, we obtained $\sigma_\star=30.03$, $44.28$, $60.61$ and $81.81$ for $k=1,2,3,4$ with the explicit lift. Analytic polynomial moments and independent quadrature give the same values.

For a given pair $v$, let $E_{\min}v$ minimize $b(w,w)$ among continuous piecewise $\Pol_{k+d+1}$ functions on the subdivision satisfying $w\in H_0^1(\Omega)$ and $Q_hw=v$. The explicit reconstruction is feasible. All feasible functions have projected gradient $\nabla_wv$, so minimizing the total energy also minimizes $\highform(v,v)$ and can only lower the threshold or leave it unchanged. In Table~\ref{tab:penalty} we compute $E_{\min}$ by a global stiffness solve with the cell and face moments as constraints. At 24 triangles the required $\sigma_\star$ falls by about $11\%$, and both the polynomial condition number and physical error decrease. The global solve couples the reconstruction across cells. We retain the explicit local lift in the refinement studies.

We tested $\sigma=\gamma_h+\tfrac12$ against $\sigma_\star=\gamma_h+1$ on all five uniform triangular meshes of Table~\ref{tab:local}, using the same lift, load and corner profile. At $L=64$ the smaller margin changes the polynomial WG error from $0.038968$ to $0.038910$ and the enriched error from $0.016784$ to $0.016760$. The smaller margin reduced both errors by less than $0.16\%$ on every mesh. In a separate test on the 24-triangle mesh, using the patch bound in place of the global quotient raised the full hat-space error from $0.149840$ to $0.162494$ (Section~\ref{sec:hats}). On the finest uniform mesh, raising $\sigma$ from about $72.25$ to $157$ raised the polynomial error from $0.038968$ to $0.044770$ (Section~\ref{sec:graded}).

On the triangular grids of Section~\ref{sec:refinement}, the chosen penalty stays near $72.25$. At $L=4,8,16$, the polynomial stiffness in the unscaled cell and edge nodal basis has spectral condition numbers $168$, $578$ and $2{,}239$ at unit penalty, and $2{,}505$, $9{,}850$ and $39{,}310$ at the chosen penalty, an increase by a factor of about 15 at $L=4$. With the mode in its original normalization, the enlarged matrix has condition numbers $2{,}798$, $19{,}083$ and $190{,}202$. As $L$ doubles from $4$ to $8$ and then to $16$, the factors are $6.82$ and $9.97$ for the enlarged matrix and $3.93$ and $3.99$ for its polynomial block at the same penalty. Let $A$ denote the polynomial stiffness at $\sigma_\star$ and $H$ the enlarged matrix. For the mixed column $\mathbf{c}_0(\psi)$, set $z=A^{-1}\mathbf{c}_0(\psi)$. Block elimination gives
\[
 H\binom{-z}{1}=\binom{0}{s_{\psi,h}},\qquad
 \lambda_{\min}(H)\le \frac{s_{\psi,h}}{1+\|z\|_2^2},
 \qquad s_{\psi,h}=s_{\psi,h}(\sigma_\star).
\]
The Rayleigh quotient on the right falls from $7.69\times10^{-3}$ at $L=8$ to $6.95\times10^{-4}$ at $L=16$ as $s_{\psi,h}$ shrinks and $\|z\|_2$ grows. Unit-energy normalization of $\psi$ fixes the mode scale and preserves the relative squared distance $s_{\psi,h}/b(\psi,\psi)$ of $\iota\psi$ from the background. Using $q_\psi/\sqrt{s_{\psi,h}}$ as the last basis vector gives $\operatorname{diag}(A,1)$ in exact arithmetic. The extreme eigenvalues of $A$ straddle one on these three meshes. The condition numbers in the orthogonalized basis are therefore $2{,}505$, $9{,}850$ and $39{,}310$. The solver obtains the same scalar pivot from the polynomial factors.

We also solved the one-mode systems at $\sigma=1$ with the reconstructed load and the same mixed column. Replacing $A$ by $A_1$ in the pivot formula gives $-0.346193$, $-0.088065$ and $-0.023813$ for $L=4,8,16$. The polynomial block is positive definite, so each bordered matrix has exactly one negative eigenvalue. The unit-penalty solutions thus fail the energy-minimization criterion even on these finite trial spaces.

\subsection{Square weak Galerkin cells at the reentrant corner}\label{sec:squares}
We repeat the manufactured corner problem of Section~\ref{sec:refinement} on squares of side $\ell=1/L$. Each square $T$ carries three total-degree $\Pol_1(T)$ coefficients and each interior Cartesian edge carries two trace coefficients. The southwest--northeast diagonal serves only to construct $E_h$ and integrate on $T$. The four Cartesian sides supply all face unknowns and all terms in the jump penalty. The weak gradient is a single constant vector on the square, obtained from its four face integrals. We checked this vector against $|T|^{-1}\int_T\nabla E_h w$ for every local basis column. Here $C_h=\{0\}$. For a continuous squarewise affine function, the tangential $x$-slope agrees across horizontal edges and vanishes on the top boundary of each column. Likewise the $y$-slope vanishes along each row by its left boundary value. Thus $J$ is positive definite on the square-cell space. The global penalty from \eqref{eq:gamma} ranges from $154.90$ to $161.74$ on the five meshes in Table~\ref{tab:square}.

\begin{table}[htbp]
\centering\small
\caption{Square-cell WG with $N_{\mathrm{dof}}$ free background coefficients. The relative global energy errors use $u=\psi+u_r$. The last two columns give the signed coefficient errors for the indicated regular remainders, with $\widehat u_r=(1+x)u_r$.}\label{tab:square}
\begin{tabular}{@{}rrr rr rr rr@{}}\toprule
&&&\multicolumn{2}{c}{Square WG}&\multicolumn{2}{c}{WG $+\psi$}&\multicolumn{2}{c}{Signed coefficient error}\\
\cmidrule(lr){4-5}\cmidrule(lr){6-7}\cmidrule(l){8-9}
$\ell$ & Cells & $N_{\mathrm{dof}}$ & $e_{\rm phys}$ & Rate & $e_{\rm phys}$ & Rate & $c_h(u_r)-1$ & $c_h(\widehat u_r)-1$\\\midrule
$1/4$ & 48 & 304 & 0.827385 & -- & 0.628168 & -- & $+1.96\times10^{-4}$ & $-1.77\times10^{-2}$\\
$1/8$ & 192 & 1,280 & 0.603141 & 0.456 & 0.428669 & 0.551 & $+2.88\times10^{-5}$ & $-1.58\times10^{-2}$\\
$1/16$ & 768 & 5,248 & 0.372332 & 0.696 & 0.227557 & 0.914 & $+3.12\times10^{-6}$ & $-1.14\times10^{-2}$\\
$1/32$ & 3,072 & 21,248 & 0.208968 & 0.833 & 0.110735 & 1.039 & $+4.00\times10^{-7}$ & $-8.41\times10^{-3}$\\
$1/64$ & 12,288 & 85,504 & 0.115232 & 0.859 & 0.054037 & 1.035 & $+6.53\times10^{-8}$ & $-6.46\times10^{-3}$\\
\bottomrule\end{tabular}
\end{table}

The single added coefficient reduces the finest-grid physical error from $0.115232$ to $0.054037$. The last enriched rate is $1.035$. The original square-grid coefficient is exceptionally close to one. To test its dependence on the load, we kept the mesh, mode, penalty and stiffness fixed and replaced $u_r$ by $\widehat u_r=(1+x)u_r$. The regular load became $(1+x)(-\Delta u_r)-2\partial_xu_r$. On the finest square grid the signed coefficient error changed from $+6.53\times10^{-8}$ to $-6.46\times10^{-3}$. For the same two remainders on the finest triangular grid, the errors were $-2.41\times10^{-2}$ and $-2.38\times10^{-2}$.

Direct gradient integration in \eqref{eq:coefficient-pairing} gives the reason for the small square-grid value. At $L=64$ the common denominator is $s_{\psi,h}=0.0115842$. Changing the remainder changes the numerator from $+7.57\times10^{-10}$ to $-7.49\times10^{-5}$. Thus the near-orthogonality of the reconstructed regular-part error to $\psi$ depends on the choice of $u_r$ on this mesh. The coefficient computed independently from the volume load agrees with this identity for both remainders on all five levels.

\subsection{Corner enrichment on meshes cut from a hexagonal tiling}\label{sec:polygons}
We next solve the original problem $u=\psi+u_r$ on a family of convex polygonal meshes. Tile the plane by translates, through $\ell\mathbb Z^2$, of the hexagon with successive vertices
\[
 \frac{\ell}{6}\bigl\{(4,2),(2,4),(-2,2),(-4,-2),(-2,-4),(2,-2)\bigr\}.
\]
Intersect the tiling with each closed half-unit square in the L-shaped domain, retaining intersections of positive area. The resulting four-, five- and six-sided cells fit both the reentrant boundary and the cutoff lines. Hexagons account for $12$ of $108$ cells at $L=4$ and $11{,}532$ of $13{,}068$ at $L=64$ (Table~\ref{tab:polygon}). Each polygon carries three affine cell coefficients and each interior polygon side carries two trace coefficients. Joining the mean of the vertices to each side gives a conforming fan for reconstruction and integration (Figure~\ref{fig:polygons}). All spokes are internal to a WG cell. The weak gradient is one constant vector per polygon, obtained from the fluxes on its sides.

For the conforming comparison, take the continuous $P_1$ space on the same fan triangulation and impose zero values on the physical boundary. Its free nodes are the interior polygon vertices and one center per polygon. Enrich this space with the same global $\psi$. Both discretizations use the same forcing, cutoff, quadrature orders and relative physical energy error. The fan method solves for nodal values on the triangulation, while WG keeps the original polygon cell and side coordinates. Figure~\ref{fig:polygons} draws the polygon edges from the computed vertex coordinates and edge incidences.

\begin{figure}[htbp]
\centering
\begin{tikzpicture}[x=0.05cm,y=0.05cm]
\draw[black!66,line width=0.23pt] (-48,-45)--(-46,-44) (-48,-39)--(-46,-38) (-48,-33)--(-46,-32) (-48,-27)--(-46,-26) (-48,-24)--(-45,-24) (-48,-21)--(-46,-20) (-48,-15)--(-46,-14) (-48,-9)--(-46,-8) (-48,-3)--(-46,-2) (-48,0)--(-45,0) (-48,3)--(-46,4) (-48,9)--(-46,10) (-48,15)--(-46,16) (-48,21)--(-46,22) (-48,24)--(-45,24) (-48,27)--(-46,28) (-48,33)--(-46,34) (-48,39)--(-46,40) (-48,45)--(-46,46) (-46,-44)--(-44,-46) (-46,-44)--(-44,-40) (-46,-38)--(-44,-40) (-46,-38)--(-44,-34) (-46,-32)--(-44,-34) (-46,-32)--(-44,-28) (-46,-26)--(-45,-24) (-46,-26)--(-44,-28) (-46,-20)--(-44,-22) (-46,-20)--(-44,-16) (-46,-14)--(-44,-16) (-46,-14)--(-44,-10) (-46,-8)--(-44,-10);
\draw[black!66,line width=0.23pt] (-46,-8)--(-44,-4) (-46,-2)--(-45,0) (-46,-2)--(-44,-4) (-46,4)--(-44,2) (-46,4)--(-44,8) (-46,10)--(-44,8) (-46,10)--(-44,14) (-46,16)--(-44,14) (-46,16)--(-44,20) (-46,22)--(-45,24) (-46,22)--(-44,20) (-46,28)--(-44,26) (-46,28)--(-44,32) (-46,34)--(-44,32) (-46,34)--(-44,38) (-46,40)--(-44,38) (-46,40)--(-44,44) (-46,46)--(-45,48) (-46,46)--(-44,44) (-45,-48)--(-44,-46) (-45,-24)--(-44,-22) (-45,-24)--(-39,-24) (-45,0)--(-44,2) (-45,0)--(-39,0) (-45,24)--(-44,26) (-45,24)--(-39,24) (-44,-46)--(-40,-44) (-44,-40)--(-40,-38) (-44,-34)--(-40,-32) (-44,-28)--(-40,-26) (-44,-22)--(-40,-20) (-44,-16)--(-40,-14);
\draw[black!66,line width=0.23pt] (-44,-10)--(-40,-8) (-44,-4)--(-40,-2) (-44,2)--(-40,4) (-44,8)--(-40,10) (-44,14)--(-40,16) (-44,20)--(-40,22) (-44,26)--(-40,28) (-44,32)--(-40,34) (-44,38)--(-40,40) (-44,44)--(-40,46) (-40,-44)--(-38,-46) (-40,-44)--(-38,-40) (-40,-38)--(-38,-40) (-40,-38)--(-38,-34) (-40,-32)--(-38,-34) (-40,-32)--(-38,-28) (-40,-26)--(-39,-24) (-40,-26)--(-38,-28) (-40,-20)--(-38,-22) (-40,-20)--(-38,-16) (-40,-14)--(-38,-16) (-40,-14)--(-38,-10) (-40,-8)--(-38,-10) (-40,-8)--(-38,-4) (-40,-2)--(-39,0) (-40,-2)--(-38,-4) (-40,4)--(-38,2) (-40,4)--(-38,8) (-40,10)--(-38,8) (-40,10)--(-38,14) (-40,16)--(-38,14) (-40,16)--(-38,20);
\draw[black!66,line width=0.23pt] (-40,22)--(-39,24) (-40,22)--(-38,20) (-40,28)--(-38,26) (-40,28)--(-38,32) (-40,34)--(-38,32) (-40,34)--(-38,38) (-40,40)--(-38,38) (-40,40)--(-38,44) (-40,46)--(-39,48) (-40,46)--(-38,44) (-39,-48)--(-38,-46) (-39,-24)--(-38,-22) (-39,-24)--(-33,-24) (-39,0)--(-38,2) (-39,0)--(-33,0) (-39,24)--(-38,26) (-39,24)--(-33,24) (-38,-46)--(-34,-44) (-38,-40)--(-34,-38) (-38,-34)--(-34,-32) (-38,-28)--(-34,-26) (-38,-22)--(-34,-20) (-38,-16)--(-34,-14) (-38,-10)--(-34,-8) (-38,-4)--(-34,-2) (-38,2)--(-34,4) (-38,8)--(-34,10) (-38,14)--(-34,16) (-38,20)--(-34,22) (-38,26)--(-34,28) (-38,32)--(-34,34) (-38,38)--(-34,40);
\draw[black!66,line width=0.23pt] (-38,44)--(-34,46) (-34,-44)--(-32,-46) (-34,-44)--(-32,-40) (-34,-38)--(-32,-40) (-34,-38)--(-32,-34) (-34,-32)--(-32,-34) (-34,-32)--(-32,-28) (-34,-26)--(-33,-24) (-34,-26)--(-32,-28) (-34,-20)--(-32,-22) (-34,-20)--(-32,-16) (-34,-14)--(-32,-16) (-34,-14)--(-32,-10) (-34,-8)--(-32,-10) (-34,-8)--(-32,-4) (-34,-2)--(-33,0) (-34,-2)--(-32,-4) (-34,4)--(-32,2) (-34,4)--(-32,8) (-34,10)--(-32,8) (-34,10)--(-32,14) (-34,16)--(-32,14) (-34,16)--(-32,20) (-34,22)--(-33,24) (-34,22)--(-32,20) (-34,28)--(-32,26) (-34,28)--(-32,32) (-34,34)--(-32,32) (-34,34)--(-32,38) (-34,40)--(-32,38) (-34,40)--(-32,44) (-34,46)--(-33,48);
\draw[black!66,line width=0.23pt] (-34,46)--(-32,44) (-33,-48)--(-32,-46) (-33,-24)--(-32,-22) (-33,-24)--(-27,-24) (-33,0)--(-32,2) (-33,0)--(-27,0) (-33,24)--(-32,26) (-33,24)--(-27,24) (-32,-46)--(-28,-44) (-32,-40)--(-28,-38) (-32,-34)--(-28,-32) (-32,-28)--(-28,-26) (-32,-22)--(-28,-20) (-32,-16)--(-28,-14) (-32,-10)--(-28,-8) (-32,-4)--(-28,-2) (-32,2)--(-28,4) (-32,8)--(-28,10) (-32,14)--(-28,16) (-32,20)--(-28,22) (-32,26)--(-28,28) (-32,32)--(-28,34) (-32,38)--(-28,40) (-32,44)--(-28,46) (-28,-44)--(-26,-46) (-28,-44)--(-26,-40) (-28,-38)--(-26,-40) (-28,-38)--(-26,-34) (-28,-32)--(-26,-34) (-28,-32)--(-26,-28) (-28,-26)--(-27,-24) (-28,-26)--(-26,-28);
\draw[black!66,line width=0.23pt] (-28,-20)--(-26,-22) (-28,-20)--(-26,-16) (-28,-14)--(-26,-16) (-28,-14)--(-26,-10) (-28,-8)--(-26,-10) (-28,-8)--(-26,-4) (-28,-2)--(-27,0) (-28,-2)--(-26,-4) (-28,4)--(-26,2) (-28,4)--(-26,8) (-28,10)--(-26,8) (-28,10)--(-26,14) (-28,16)--(-26,14) (-28,16)--(-26,20) (-28,22)--(-27,24) (-28,22)--(-26,20) (-28,28)--(-26,26) (-28,28)--(-26,32) (-28,34)--(-26,32) (-28,34)--(-26,38) (-28,40)--(-26,38) (-28,40)--(-26,44) (-28,46)--(-27,48) (-28,46)--(-26,44) (-27,-48)--(-26,-46) (-27,-24)--(-26,-22) (-27,-24)--(-24,-24) (-27,0)--(-26,2) (-27,0)--(-24,0) (-27,24)--(-26,26) (-27,24)--(-24,24) (-26,-46)--(-24,-45);
\draw[black!66,line width=0.23pt] (-26,-40)--(-24,-39) (-26,-34)--(-24,-33) (-26,-28)--(-24,-27) (-26,-22)--(-24,-21) (-26,-16)--(-24,-15) (-26,-10)--(-24,-9) (-26,-4)--(-24,-3) (-26,2)--(-24,3) (-26,8)--(-24,9) (-26,14)--(-24,15) (-26,20)--(-24,21) (-26,26)--(-24,27) (-26,32)--(-24,33) (-26,38)--(-24,39) (-26,44)--(-24,45) (-24,-48)--(-24,-45) (-24,-45)--(-24,-39) (-24,-45)--(-22,-44) (-24,-39)--(-24,-33) (-24,-39)--(-22,-38) (-24,-33)--(-24,-27) (-24,-33)--(-22,-32) (-24,-27)--(-24,-24) (-24,-27)--(-22,-26) (-24,-24)--(-24,-21) (-24,-24)--(-21,-24) (-24,-21)--(-24,-15) (-24,-21)--(-22,-20) (-24,-15)--(-24,-9) (-24,-15)--(-22,-14) (-24,-9)--(-24,-3) (-24,-9)--(-22,-8);
\draw[black!66,line width=0.23pt] (-24,-3)--(-24,0) (-24,-3)--(-22,-2) (-24,0)--(-24,3) (-24,0)--(-21,0) (-24,3)--(-24,9) (-24,3)--(-22,4) (-24,9)--(-24,15) (-24,9)--(-22,10) (-24,15)--(-24,21) (-24,15)--(-22,16) (-24,21)--(-24,24) (-24,21)--(-22,22) (-24,24)--(-24,27) (-24,24)--(-21,24) (-24,27)--(-24,33) (-24,27)--(-22,28) (-24,33)--(-24,39) (-24,33)--(-22,34) (-24,39)--(-24,45) (-24,39)--(-22,40) (-24,45)--(-24,48) (-24,45)--(-22,46) (-22,-44)--(-20,-46) (-22,-44)--(-20,-40) (-22,-38)--(-20,-40) (-22,-38)--(-20,-34) (-22,-32)--(-20,-34) (-22,-32)--(-20,-28) (-22,-26)--(-21,-24) (-22,-26)--(-20,-28) (-22,-20)--(-20,-22) (-22,-20)--(-20,-16);
\draw[black!66,line width=0.23pt] (-22,-14)--(-20,-16) (-22,-14)--(-20,-10) (-22,-8)--(-20,-10) (-22,-8)--(-20,-4) (-22,-2)--(-21,0) (-22,-2)--(-20,-4) (-22,4)--(-20,2) (-22,4)--(-20,8) (-22,10)--(-20,8) (-22,10)--(-20,14) (-22,16)--(-20,14) (-22,16)--(-20,20) (-22,22)--(-21,24) (-22,22)--(-20,20) (-22,28)--(-20,26) (-22,28)--(-20,32) (-22,34)--(-20,32) (-22,34)--(-20,38) (-22,40)--(-20,38) (-22,40)--(-20,44) (-22,46)--(-21,48) (-22,46)--(-20,44) (-21,-48)--(-20,-46) (-21,-24)--(-20,-22) (-21,-24)--(-15,-24) (-21,0)--(-20,2) (-21,0)--(-15,0) (-21,24)--(-20,26) (-21,24)--(-15,24) (-20,-46)--(-16,-44) (-20,-40)--(-16,-38) (-20,-34)--(-16,-32);
\draw[black!66,line width=0.23pt] (-20,-28)--(-16,-26) (-20,-22)--(-16,-20) (-20,-16)--(-16,-14) (-20,-10)--(-16,-8) (-20,-4)--(-16,-2) (-20,2)--(-16,4) (-20,8)--(-16,10) (-20,14)--(-16,16) (-20,20)--(-16,22) (-20,26)--(-16,28) (-20,32)--(-16,34) (-20,38)--(-16,40) (-20,44)--(-16,46) (-16,-44)--(-14,-46) (-16,-44)--(-14,-40) (-16,-38)--(-14,-40) (-16,-38)--(-14,-34) (-16,-32)--(-14,-34) (-16,-32)--(-14,-28) (-16,-26)--(-15,-24) (-16,-26)--(-14,-28) (-16,-20)--(-14,-22) (-16,-20)--(-14,-16) (-16,-14)--(-14,-16) (-16,-14)--(-14,-10) (-16,-8)--(-14,-10) (-16,-8)--(-14,-4) (-16,-2)--(-15,0) (-16,-2)--(-14,-4) (-16,4)--(-14,2) (-16,4)--(-14,8) (-16,10)--(-14,8);
\draw[black!66,line width=0.23pt] (-16,10)--(-14,14) (-16,16)--(-14,14) (-16,16)--(-14,20) (-16,22)--(-15,24) (-16,22)--(-14,20) (-16,28)--(-14,26) (-16,28)--(-14,32) (-16,34)--(-14,32) (-16,34)--(-14,38) (-16,40)--(-14,38) (-16,40)--(-14,44) (-16,46)--(-15,48) (-16,46)--(-14,44) (-15,-48)--(-14,-46) (-15,-24)--(-14,-22) (-15,-24)--(-9,-24) (-15,0)--(-14,2) (-15,0)--(-9,0) (-15,24)--(-14,26) (-15,24)--(-9,24) (-14,-46)--(-10,-44) (-14,-40)--(-10,-38) (-14,-34)--(-10,-32) (-14,-28)--(-10,-26) (-14,-22)--(-10,-20) (-14,-16)--(-10,-14) (-14,-10)--(-10,-8) (-14,-4)--(-10,-2) (-14,2)--(-10,4) (-14,8)--(-10,10) (-14,14)--(-10,16) (-14,20)--(-10,22);
\draw[black!66,line width=0.23pt] (-14,26)--(-10,28) (-14,32)--(-10,34) (-14,38)--(-10,40) (-14,44)--(-10,46) (-10,-44)--(-8,-46) (-10,-44)--(-8,-40) (-10,-38)--(-8,-40) (-10,-38)--(-8,-34) (-10,-32)--(-8,-34) (-10,-32)--(-8,-28) (-10,-26)--(-9,-24) (-10,-26)--(-8,-28) (-10,-20)--(-8,-22) (-10,-20)--(-8,-16) (-10,-14)--(-8,-16) (-10,-14)--(-8,-10) (-10,-8)--(-8,-10) (-10,-8)--(-8,-4) (-10,-2)--(-9,0) (-10,-2)--(-8,-4) (-10,4)--(-8,2) (-10,4)--(-8,8) (-10,10)--(-8,8) (-10,10)--(-8,14) (-10,16)--(-8,14) (-10,16)--(-8,20) (-10,22)--(-9,24) (-10,22)--(-8,20) (-10,28)--(-8,26) (-10,28)--(-8,32) (-10,34)--(-8,32) (-10,34)--(-8,38);
\draw[black!66,line width=0.23pt] (-10,40)--(-8,38) (-10,40)--(-8,44) (-10,46)--(-9,48) (-10,46)--(-8,44) (-9,-48)--(-8,-46) (-9,-24)--(-8,-22) (-9,-24)--(-3,-24) (-9,0)--(-8,2) (-9,0)--(-3,0) (-9,24)--(-8,26) (-9,24)--(-3,24) (-8,-46)--(-4,-44) (-8,-40)--(-4,-38) (-8,-34)--(-4,-32) (-8,-28)--(-4,-26) (-8,-22)--(-4,-20) (-8,-16)--(-4,-14) (-8,-10)--(-4,-8) (-8,-4)--(-4,-2) (-8,2)--(-4,4) (-8,8)--(-4,10) (-8,14)--(-4,16) (-8,20)--(-4,22) (-8,26)--(-4,28) (-8,32)--(-4,34) (-8,38)--(-4,40) (-8,44)--(-4,46) (-4,-44)--(-2,-46) (-4,-44)--(-2,-40) (-4,-38)--(-2,-40) (-4,-38)--(-2,-34) (-4,-32)--(-2,-34);
\draw[black!66,line width=0.23pt] (-4,-32)--(-2,-28) (-4,-26)--(-3,-24) (-4,-26)--(-2,-28) (-4,-20)--(-2,-22) (-4,-20)--(-2,-16) (-4,-14)--(-2,-16) (-4,-14)--(-2,-10) (-4,-8)--(-2,-10) (-4,-8)--(-2,-4) (-4,-2)--(-3,0) (-4,-2)--(-2,-4) (-4,4)--(-2,2) (-4,4)--(-2,8) (-4,10)--(-2,8) (-4,10)--(-2,14) (-4,16)--(-2,14) (-4,16)--(-2,20) (-4,22)--(-3,24) (-4,22)--(-2,20) (-4,28)--(-2,26) (-4,28)--(-2,32) (-4,34)--(-2,32) (-4,34)--(-2,38) (-4,40)--(-2,38) (-4,40)--(-2,44) (-4,46)--(-3,48) (-4,46)--(-2,44) (-3,-48)--(-2,-46) (-3,-24)--(-2,-22) (-3,-24)--(0,-24) (-3,0)--(-2,2) (-3,0)--(0,0);
\draw[black!66,line width=0.23pt] (-3,24)--(-2,26) (-3,24)--(0,24) (-2,-46)--(0,-45) (-2,-40)--(0,-39) (-2,-34)--(0,-33) (-2,-28)--(0,-27) (-2,-22)--(0,-21) (-2,-16)--(0,-15) (-2,-10)--(0,-9) (-2,-4)--(0,-3) (-2,2)--(0,3) (-2,8)--(0,9) (-2,14)--(0,15) (-2,20)--(0,21) (-2,26)--(0,27) (-2,32)--(0,33) (-2,38)--(0,39) (-2,44)--(0,45) (0,0)--(0,3) (0,3)--(0,9) (0,3)--(2,4) (0,9)--(0,15) (0,9)--(2,10) (0,15)--(0,21) (0,15)--(2,16) (0,21)--(0,24) (0,21)--(2,22) (0,24)--(0,27) (0,24)--(3,24) (0,27)--(0,33) (0,27)--(2,28) (0,33)--(0,39);
\draw[black!66,line width=0.23pt] (0,33)--(2,34) (0,39)--(0,45) (0,39)--(2,40) (0,45)--(0,48) (0,45)--(2,46) (2,4)--(4,2) (2,4)--(4,8) (2,10)--(4,8) (2,10)--(4,14) (2,16)--(4,14) (2,16)--(4,20) (2,22)--(3,24) (2,22)--(4,20) (2,28)--(4,26) (2,28)--(4,32) (2,34)--(4,32) (2,34)--(4,38) (2,40)--(4,38) (2,40)--(4,44) (2,46)--(3,48) (2,46)--(4,44) (3,0)--(4,2) (3,24)--(4,26) (3,24)--(9,24) (4,2)--(8,4) (4,8)--(8,10) (4,14)--(8,16) (4,20)--(8,22) (4,26)--(8,28) (4,32)--(8,34) (4,38)--(8,40) (4,44)--(8,46);
\draw[black!66,line width=0.23pt] (8,4)--(10,2) (8,4)--(10,8) (8,10)--(10,8) (8,10)--(10,14) (8,16)--(10,14) (8,16)--(10,20) (8,22)--(9,24) (8,22)--(10,20) (8,28)--(10,26) (8,28)--(10,32) (8,34)--(10,32) (8,34)--(10,38) (8,40)--(10,38) (8,40)--(10,44) (8,46)--(9,48) (8,46)--(10,44) (9,0)--(10,2) (9,24)--(10,26) (9,24)--(15,24) (10,2)--(14,4) (10,8)--(14,10) (10,14)--(14,16) (10,20)--(14,22) (10,26)--(14,28) (10,32)--(14,34) (10,38)--(14,40) (10,44)--(14,46) (14,4)--(16,2) (14,4)--(16,8) (14,10)--(16,8) (14,10)--(16,14) (14,16)--(16,14);
\draw[black!66,line width=0.23pt] (14,16)--(16,20) (14,22)--(15,24) (14,22)--(16,20) (14,28)--(16,26) (14,28)--(16,32) (14,34)--(16,32) (14,34)--(16,38) (14,40)--(16,38) (14,40)--(16,44) (14,46)--(15,48) (14,46)--(16,44) (15,0)--(16,2) (15,24)--(16,26) (15,24)--(21,24) (16,2)--(20,4) (16,8)--(20,10) (16,14)--(20,16) (16,20)--(20,22) (16,26)--(20,28) (16,32)--(20,34) (16,38)--(20,40) (16,44)--(20,46) (20,4)--(22,2) (20,4)--(22,8) (20,10)--(22,8) (20,10)--(22,14) (20,16)--(22,14) (20,16)--(22,20) (20,22)--(21,24) (20,22)--(22,20) (20,28)--(22,26) (20,28)--(22,32);
\draw[black!66,line width=0.23pt] (20,34)--(22,32) (20,34)--(22,38) (20,40)--(22,38) (20,40)--(22,44) (20,46)--(21,48) (20,46)--(22,44) (21,0)--(22,2) (21,24)--(22,26) (21,24)--(24,24) (22,2)--(24,3) (22,8)--(24,9) (22,14)--(24,15) (22,20)--(24,21) (22,26)--(24,27) (22,32)--(24,33) (22,38)--(24,39) (22,44)--(24,45) (24,0)--(24,3) (24,3)--(24,9) (24,3)--(26,4) (24,9)--(24,15) (24,9)--(26,10) (24,15)--(24,21) (24,15)--(26,16) (24,21)--(24,24) (24,21)--(26,22) (24,24)--(24,27) (24,24)--(27,24) (24,27)--(24,33) (24,27)--(26,28) (24,33)--(24,39) (24,33)--(26,34);
\draw[black!66,line width=0.23pt] (24,39)--(24,45) (24,39)--(26,40) (24,45)--(24,48) (24,45)--(26,46) (26,4)--(28,2) (26,4)--(28,8) (26,10)--(28,8) (26,10)--(28,14) (26,16)--(28,14) (26,16)--(28,20) (26,22)--(27,24) (26,22)--(28,20) (26,28)--(28,26) (26,28)--(28,32) (26,34)--(28,32) (26,34)--(28,38) (26,40)--(28,38) (26,40)--(28,44) (26,46)--(27,48) (26,46)--(28,44) (27,0)--(28,2) (27,24)--(28,26) (27,24)--(33,24) (28,2)--(32,4) (28,8)--(32,10) (28,14)--(32,16) (28,20)--(32,22) (28,26)--(32,28) (28,32)--(32,34) (28,38)--(32,40) (28,44)--(32,46) (32,4)--(34,2);
\draw[black!66,line width=0.23pt] (32,4)--(34,8) (32,10)--(34,8) (32,10)--(34,14) (32,16)--(34,14) (32,16)--(34,20) (32,22)--(33,24) (32,22)--(34,20) (32,28)--(34,26) (32,28)--(34,32) (32,34)--(34,32) (32,34)--(34,38) (32,40)--(34,38) (32,40)--(34,44) (32,46)--(33,48) (32,46)--(34,44) (33,0)--(34,2) (33,24)--(34,26) (33,24)--(39,24) (34,2)--(38,4) (34,8)--(38,10) (34,14)--(38,16) (34,20)--(38,22) (34,26)--(38,28) (34,32)--(38,34) (34,38)--(38,40) (34,44)--(38,46) (38,4)--(40,2) (38,4)--(40,8) (38,10)--(40,8) (38,10)--(40,14) (38,16)--(40,14) (38,16)--(40,20);
\draw[black!66,line width=0.23pt] (38,22)--(39,24) (38,22)--(40,20) (38,28)--(40,26) (38,28)--(40,32) (38,34)--(40,32) (38,34)--(40,38) (38,40)--(40,38) (38,40)--(40,44) (38,46)--(39,48) (38,46)--(40,44) (39,0)--(40,2) (39,24)--(40,26) (39,24)--(45,24) (40,2)--(44,4) (40,8)--(44,10) (40,14)--(44,16) (40,20)--(44,22) (40,26)--(44,28) (40,32)--(44,34) (40,38)--(44,40) (40,44)--(44,46) (44,4)--(46,2) (44,4)--(46,8) (44,10)--(46,8) (44,10)--(46,14) (44,16)--(46,14) (44,16)--(46,20) (44,22)--(45,24) (44,22)--(46,20) (44,28)--(46,26) (44,28)--(46,32) (44,34)--(46,32);
\draw[black!66,line width=0.23pt] (44,34)--(46,38) (44,40)--(46,38) (44,40)--(46,44) (44,46)--(45,48) (44,46)--(46,44) (45,0)--(46,2) (45,24)--(46,26) (45,24)--(48,24) (46,2)--(48,3) (46,8)--(48,9) (46,14)--(48,15) (46,20)--(48,21) (46,26)--(48,27) (46,32)--(48,33) (46,38)--(48,39) (46,44)--(48,45);
\draw[black,line width=0.65pt] (-48,-48)--(-48,-45) (-48,-48)--(-45,-48) (-48,-45)--(-48,-39) (-48,-39)--(-48,-33) (-48,-33)--(-48,-27) (-48,-27)--(-48,-24) (-48,-24)--(-48,-21) (-48,-21)--(-48,-15) (-48,-15)--(-48,-9) (-48,-9)--(-48,-3) (-48,-3)--(-48,0) (-48,0)--(-48,3) (-48,3)--(-48,9) (-48,9)--(-48,15) (-48,15)--(-48,21) (-48,21)--(-48,24) (-48,24)--(-48,27) (-48,27)--(-48,33) (-48,33)--(-48,39) (-48,39)--(-48,45) (-48,45)--(-48,48) (-48,48)--(-45,48) (-45,-48)--(-39,-48) (-45,48)--(-39,48) (-39,-48)--(-33,-48) (-39,48)--(-33,48) (-33,-48)--(-27,-48) (-33,48)--(-27,48) (-27,-48)--(-24,-48) (-27,48)--(-24,48) (-24,-48)--(-21,-48) (-24,48)--(-21,48);
\draw[black,line width=0.65pt] (-21,-48)--(-15,-48) (-21,48)--(-15,48) (-15,-48)--(-9,-48) (-15,48)--(-9,48) (-9,-48)--(-3,-48) (-9,48)--(-3,48) (-3,-48)--(0,-48) (-3,48)--(0,48) (0,-48)--(0,-45) (0,-45)--(0,-39) (0,-39)--(0,-33) (0,-33)--(0,-27) (0,-27)--(0,-24) (0,-24)--(0,-21) (0,-21)--(0,-15) (0,-15)--(0,-9) (0,-9)--(0,-3) (0,-3)--(0,0) (0,0)--(3,0) (0,48)--(3,48) (3,0)--(9,0) (3,48)--(9,48) (9,0)--(15,0) (9,48)--(15,48) (15,0)--(21,0) (15,48)--(21,48) (21,0)--(24,0) (21,48)--(24,48) (24,0)--(27,0) (24,48)--(27,48) (27,0)--(33,0) (27,48)--(33,48);
\draw[black,line width=0.65pt] (33,0)--(39,0) (33,48)--(39,48) (39,0)--(45,0) (39,48)--(45,48) (45,0)--(48,0) (45,48)--(48,48) (48,0)--(48,3) (48,3)--(48,9) (48,9)--(48,15) (48,15)--(48,21) (48,21)--(48,24) (48,24)--(48,27) (48,27)--(48,33) (48,33)--(48,39) (48,39)--(48,45) (48,45)--(48,48);
\draw[black,line width=0.65pt,dashed] (-6,-6) rectangle (6,6);
\end{tikzpicture}
\hspace{1.7em}
\begin{tikzpicture}[x=0.20cm,y=0.20cm]
\clip (-12,-12) rectangle (12,12);
\fill[black!17] (-2,8)--(-4,10)--(-8,8)--(-10,4)--(-8,2)--(-4,4)--cycle;
\draw[black!70,line width=0.35pt] (-15,0)--(-9,0) (-14,-10)--(-10,-8) (-14,-4)--(-10,-2) (-14,2)--(-10,4) (-14,8)--(-10,10) (-10,-14)--(-8,-10) (-10,-8)--(-8,-10) (-10,-8)--(-8,-4) (-10,-2)--(-9,0) (-10,-2)--(-8,-4) (-10,4)--(-8,2) (-10,4)--(-8,8) (-10,10)--(-8,8) (-10,10)--(-8,14) (-9,0)--(-8,2) (-9,0)--(-3,0) (-8,-10)--(-4,-8) (-8,-4)--(-4,-2) (-8,2)--(-4,4) (-8,8)--(-4,10) (-4,-14)--(-2,-10) (-4,-8)--(-2,-10) (-4,-8)--(-2,-4) (-4,-2)--(-3,0) (-4,-2)--(-2,-4) (-4,4)--(-2,2) (-4,4)--(-2,8) (-4,10)--(-2,8) (-4,10)--(-2,14) (-3,0)--(-2,2) (-3,0)--(0,0) (-2,-10)--(0,-9);
\draw[black!70,line width=0.35pt] (-2,-4)--(0,-3) (-2,2)--(0,3) (-2,8)--(0,9) (0,0)--(0,3) (0,3)--(0,9) (0,3)--(2,4) (0,9)--(0,15) (0,9)--(2,10) (2,4)--(4,2) (2,4)--(4,8) (2,10)--(4,8) (2,10)--(4,14) (3,0)--(4,2) (4,2)--(8,4) (4,8)--(8,10) (8,4)--(10,2) (8,4)--(10,8) (8,10)--(10,8) (8,10)--(10,14) (9,0)--(10,2) (10,2)--(14,4) (10,8)--(14,10);
\draw[black,line width=0.65pt] (0,-15)--(0,-9) (0,-9)--(0,-3) (0,-3)--(0,0) (0,0)--(3,0) (3,0)--(9,0) (9,0)--(15,0);
\draw[black,dash pattern=on 2pt off 1.25pt,line width=0.68pt] (-6,6)--(-2,8) (-6,6)--(-4,10) (-6,6)--(-8,8) (-6,6)--(-10,4) (-6,6)--(-8,2) (-6,6)--(-4,4);
\fill[black] (-6,6) circle[radius=0.85pt];
\end{tikzpicture}
\space
\caption{The dashed box on the $L=8$ mesh (left) marks the window $[-1/8,1/8]^2$ on the $L=16$ mesh (right). Solid polygon edges carry the WG trace. The six dashed spokes join the marked center of the shaded hexagon to its vertices and define its reconstruction triangles.}\label{fig:polygons}
\end{figure}

\begin{table}[htbp]
\centering\small
\caption{One prescribed corner mode on the same polygons. Each background has $N_{\mathrm{dof}}$ free coefficients and each enriched solve has $N_{\mathrm{dof}}+1$. The two WG solves use the same penalty. Errors are relative global physical energy errors.}\label{tab:polygon}
\begin{tabular}{@{}r r r r rr rr r@{}}\toprule
&&&&\multicolumn{2}{c}{Polygon WG}&\multicolumn{2}{c}{WG $+\psi$}&\\
\cmidrule(lr){5-6}\cmidrule(lr){7-8}
$L$ & Cells & Hexagons & $N_{\mathrm{dof}}$ & $e_{\rm phys}$ & Rate & $e_{\rm phys}$ & Rate & $c_h$\\\midrule
4 & 108 & 12 & 804 & 0.668011 & -- & 0.481500 & -- & 0.920138\\
8 & 300 & 108 & 2,404 & 0.505621 & 0.402 & 0.332440 & 0.534 & 0.911095\\
16 & 972 & 588 & 8,196 & 0.330456 & 0.614 & 0.194521 & 0.773 & 0.922703\\
32 & 3,468 & 2,700 & 30,148 & 0.195665 & 0.756 & 0.104711 & 0.894 & 0.938321\\
64 & 13,068 & 11,532 & 115,524 & 0.110865 & 0.820 & 0.054214 & 0.950 & 0.951583\\
\midrule
&&&&\multicolumn{2}{c}{Fan $P_1$}&\multicolumn{2}{c}{Fan $P_1+\psi$}&\\
\cmidrule(lr){5-6}\cmidrule(lr){7-8}
$L$ & \multicolumn{2}{c}{Fan triangles} & $N_{\mathrm{dof}}$ & $e_{\rm phys}$ & Rate & $e_{\rm phys}$ & Rate & $c_h$\\\midrule
4 & \multicolumn{2}{c}{528} & 241 & 0.286236 & -- & 0.144621 & -- & 0.977747\\
8 & \multicolumn{2}{c}{1,584} & 753 & 0.175035 & 0.710 & 0.078902 & 0.874 & 0.980534\\
16 & \multicolumn{2}{c}{5,424} & 2,641 & 0.094322 & 0.892 & 0.040480 & 0.963 & 0.981813\\
32 & \multicolumn{2}{c}{20,016} & 9,873 & 0.051772 & 0.865 & 0.020410 & 0.988 & 0.984764\\
64 & \multicolumn{2}{c}{76,848} & 38,161 & 0.029368 & 0.818 & 0.010236 & 0.996 & 0.988369\\
\bottomrule\end{tabular}
\end{table}

The global penalty $1+\gamma_h$ ranges from $204.91$ to $236.57$. A zero-jump function is determined by its interior vertex values, which must be affine on each polygon. The integer matrix of these constraints has full column rank modulo the prime $65{,}521$ on each of the five meshes. A maximal minor is therefore nonzero over the reals, giving $C_h=\{0\}$. The reconstruction subdivision supports a conforming $P_1$ solve. A fan function can change gradient across an internal spoke and thus lie outside $C_h$, whose functions are affine on each whole polygon. The full fan space therefore extends beyond the zero-jump polynomial competitors in Theorem~\ref{thm:best}. On the finest mesh the mode reduces the WG error from $0.110865$ to $0.054214$ with $115{,}525$ total coefficients. Without enrichment, fan $P_1$ already gives $0.029368$ with $38{,}161$ coefficients. Adding the same profile gives $0.010236$ with $38{,}162$. The last enriched rates are $0.950$ and $0.996$, respectively. The conforming fan gives smaller errors with fewer unknowns on all five levels.

\subsection{Corner grading at matched WG dimensions}\label{sec:graded}
We compare two ways to resolve the corner in the problem of Section~\ref{sec:refinement}. The enriched uniform mesh uses a prescribed mode $\psi$, including its singular exponent, angular dependence and cutoff. The graded polynomial mesh uses the corner location and a fixed edge-split fraction, following the idea of refinement toward a singular vertex \cite{BKP}. Start with the 24 triangles of Section~\ref{sec:hats}. In each refinement, split every edge incident to the origin one quarter of its length from the origin and bisect every other edge. Joining the three split points in each triangle gives four children and preserves the cutoff interfaces. After $j$ refinements, set $L=2^{j+1}$. Both meshes have $6L^2$ triangles and $N_{\mathrm{dof}}=36L^2-8L$ free polynomial coefficients before elimination. Each enriched solve has $N_{\mathrm{dof}}+1$ coefficients.

Table~\ref{tab:graded} uses the same forcing, energy normalization, reconstruction rule and reconstructed load on both mesh families. On each mesh the polynomial and enriched solves share its penalty $1+\gamma_h$. The last refinement gives observed rates of $0.988$ for graded polynomial WG and $0.999$ for uniform WG with $\psi$. Both approach first order in $L^{-1}$ on the tested levels. On the finest pair, the graded polynomial error is $2.26$ times the enriched uniform error, with $146{,}944$ and $146{,}945$ free coefficients, respectively.

\begin{table}[htbp]
\centering\small
\caption{Relative physical energy errors. Polynomial WG uses $N_{\mathrm{dof}}$ coefficients and WG with $\psi$ uses $N_{\mathrm{dof}}+1$. The rate $\log_2(e_{L/2}/e_L)$ is reported for uniform WG with $\psi$ and corner-graded polynomial WG. The last column gives the graded-mesh penalty.}\label{tab:graded}
\setlength{\tabcolsep}{4pt}
\begin{tabular}{@{}rr rrr rrr r@{}}\toprule
&&\multicolumn{3}{c}{Uniform}&\multicolumn{3}{c}{Corner-graded}&\\
\cmidrule(lr){3-5}\cmidrule(lr){6-8}
$L$ & $N_{\mathrm{dof}}$ & WG & WG $+\psi$ & Rate & WG & Rate & WG $+\psi$ & $\sigma_\star$\\\midrule
4 & 544 & 0.439486 & 0.258121 & -- & 0.433505 & -- & 0.275220 & 107.90\\
8 & 2,240 & 0.233565 & 0.132849 & 0.958 & 0.282443 & 0.618 & 0.149061 & 147.93\\
16 & 9,088 & 0.125978 & 0.066937 & 0.989 & 0.147233 & 0.940 & 0.075651 & 153.24\\
32 & 36,608 & 0.069207 & 0.033542 & 0.997 & 0.075108 & 0.971 & 0.038000 & 155.51\\
64 & 146,944 & 0.038968 & 0.016784 & 0.999 & 0.037874 & 0.988 & 0.019026 & 156.32\\
\bottomrule\end{tabular}
\end{table}

For a common-penalty check on the finest level, round the larger of the two admissible penalties up to $157$ and use this value on both meshes. The polynomial errors are $0.044770$ (uniform) and $0.037894$ (graded). With the mode they fall to $0.018918$ and $0.019035$, respectively. At this common penalty the graded polynomial error is $2.00$ times the enriched uniform error. The two enriched errors differ by less than one percent. These ratios depend on the supplied mode and on the chosen grading fraction. The refinement data support first-order behavior for both strategies.

\subsection{Quadratic elements at fitted material interfaces}\label{sec:transmission}
Kellogg \cite{KelloggI} studied elliptic regularity where material interfaces intersect. In the two tests here, a single interface meets a reentrant boundary corner, and the diffusion jump changes the corner profile. On the same L-shaped domain let $a=\kappa I$, with $\kappa_+=1+\sqrt{2}$ in $0<\theta<\theta_I=\pi/2$ and $\kappa_-=1$ in $\theta_I<\theta<\omega=3\pi/2$. Put $\lambda=3/4$ and $s=r^\lambda\Phi(\theta)$, where
\[
 \Phi(\theta)=
 \begin{cases}
 \sin(\lambda\theta)/\sin(\lambda\theta_I),&0\le\theta\le\theta_I,\\
 \sin(\lambda(\omega-\theta))/\sin(\lambda(\omega-\theta_I)),&
                       \theta_I<\theta\le\omega .
 \end{cases}
\]
Write $\Phi_\nu$ for this angular function with $\lambda$ replaced by $\nu$, so that the exact angular profile is $\Phi_\lambda$. The two traces agree at $\theta_I$. The identity
$\kappa_+\cot(\lambda\theta_I)+\kappa_-\cot(\lambda(\omega-\theta_I))=0$
makes $\kappa\Phi'$ continuous there, with nonzero common value. Thus $s$ is harmonic in each material and transmits flux across the interface.

Take $\chi_I=(1-x^2)(1-y^2)$, $\psi=\chi_I s$ and $u=\psi+u_r$ with $u_r=\chi_I xy/\kappa$. Both terms vanish on the domain boundary. On the interface $x=0$, $y>0$, we have $\partial_x\chi_I=0$, and $u_r$ has zero trace with $\kappa\partial_xu_r=\partial_x(\chi_I xy)$ on both sides. The volume load is therefore
\[
 f=-\kappa\bigl((\Delta\chi_I)s+2\nabla\chi_I\cdot\nabla s\bigr)+6xy(2-x^2-y^2).
\]
This volume load belongs to $L^2(\Omega)$. As in Section~\ref{sec:corner-approx}, the flux contribution on the arc $r=\varepsilon$ is $O(\varepsilon^{3/4})$ against a smooth test function and vanishes as $\varepsilon\to0$. Since $u_r$ is polynomial on each material, Theorem~\ref{thm:singular} gives an $O(h^2)$ energy bound for quadratic WG with this one mode.

The uniform triangles of Section~\ref{sec:refinement} fit the interface. For each mesh and material coefficient we compute $\sigma_\star=1+\gamma_h$ from the quadratic WG reconstruction. The polynomial solve, prescribed-mode solve and candidate scores on that mesh use the same penalty. Table~\ref{tab:transmission} lists the values for both interfaces. Both conforming $P_2$ solves use the same mesh, forcing and homogeneous boundary data as WG. The enriched conforming solve adds the same $\psi$ to its nodal basis. On a fixed triangular mesh the reconstruction reproduces $C_h$ with zero jump. Theorem~\ref{thm:best} then gives
\[
  \en{u-u_1^c}\le \mathcal E_1
  \le \min_{w\in C_h+\Span\{\psi\}}\en{u-w}.
\]
The minimum on the right is the physical energy error of enriched conforming $P_2$. Table~\ref{tab:transmission} divides all errors by $\en u$ and lists both polynomial backgrounds. Comparing the two enriched solves on each of the five vertical-interface meshes, WG has the smaller error. Both enriched methods reach rate $1.997$ on the last refinement.

Comparing by the number of free coefficients gives a different ordering across levels. At $L=32$, enriched $P_2$ uses $12{,}034$ coefficients and gives relative error $2.8406\times10^{-4}$. At $L=16$, enriched WG uses $15{,}937$ coefficients and gives $9.8900\times10^{-4}$. Both counts include the profile amplitude.

For the selection test we use $n_\nu=\chi_I r^\nu\Phi_\nu(\theta)$. The six trial exponents are $0.50$, $0.60$, $0.70$, $0.75$, $0.80$ and $0.90$. The profiles share the same interface geometry and smooth multiplier $\chi_I$. Each has continuous trace and lies in $H_0^1(\Omega)$, even when its flux jumps across the interface. On the two boundary rays of the corner, the outward normal derivatives are
\[
 \partial_n n_\nu(r,0)=-\frac{\nu\,\chi_I(r,0)}{\sin(\nu\theta_I)}\,r^{\nu-1},\qquad
 \partial_n n_\nu(r,\omega)=-\frac{\nu\,\chi_I(r,\omega)}{\sin(\nu(\omega-\theta_I))}\,r^{\nu-1}.
\]
For $\nu=1/2$, both denominators are nonzero and $\chi_I\to1$ at the vertex, so the squared normal derivative on each ray has the logarithmic divergence of $\int_0^\rho r^{-1}\dd r$. The volume energy is finite because near the vertex it is bounded by a constant times $\int_0^\rho r^{2\nu-1}\dd r$. This trial profile therefore enters our residual search although it falls outside assumption (A3) of the X-HHO analysis \cite{Yemm}.

We maximize $g_0$ using the assembled stiffness and volume load, then solve for the selected amplitude. With all six candidates, the exact exponent $0.75$ has the highest score on all five meshes. When it is withheld, the two adjacent candidates $0.70$ and $0.80$ rank first and second in either order on every mesh (Table~\ref{tab:transmission}). At $L=8$ the score for $0.80$ exceeds that for $0.70$ by $0.071\%$ of the higher score. At $L=16$ the order reverses, with a margin of $0.154\%$. Both quadrature orders give these rankings. By Theorem~\ref{thm:gain}, $g_0^2$ measures the decrease in squared method error relative to the current polynomial solution. Refinement changes this solution and the projected trial directions, so the two neighboring exponents can exchange rank.

For the second interface take $\theta_I=\pi/4$, $\kappa_+=3-2\sqrt2$, $\kappa_-=1$ and $\lambda=1/2$. The grid diagonals fit the ray $x=y>0$. Since $\kappa_+\cot(\pi/8)+\cot(5\pi/8)=0$, the positive angular profile $\Phi_{1/2}$ transmits flux. Set $\psi=\chi_I r^{1/2}\Phi_{1/2}$ and $u=\psi+\chi_I xy$. Both $\chi_I$ and the polynomial remainder are symmetric in $x,y$, so their normal derivatives vanish on the interface. The piecewise volume load is $f=-\kappa((\Delta\chi_I)s+2\nabla\chi_I\cdot\nabla s)+6\kappa xy(2-x^2-y^2)$, with $s=r^{1/2}\Phi_{1/2}$. It belongs to $L^2(\Omega)$, and the flux against smooth tests on $r=\varepsilon$ is $O(\varepsilon^{1/2})$. Thus the term with non-square-integrable normal derivative is part of the exact weak solution. We prescribe this mode in both quadratic spaces. At $L=32$, one additional coefficient reduces the relative physical energy error of quadratic WG from $0.035192$ to $3.2193\times10^{-4}$. On the finest pair the enriched WG and conforming rates are $1.999$ and $1.997$, respectively.

On each mesh the two problems share the reconstruction, the projections $\Pi_T$ and the jump form $J$. The vertical-interface diffusivity is at least one on every cell and the diagonal-interface diffusivity is at most one. Equations~\eqref{eq:K} and~\eqref{eq:gamma} then give $\sigma_{\star,\mathrm{diag}}\le\sigma_{\star,\mathrm{vert}}$ by comparison of the two Rayleigh quotients. To examine the variation with refinement, we evaluated the numerator of \eqref{eq:gamma} with $a=I$ on the computed maximizing vector for each material problem. From $L=2$ to $32$, the fraction of this unweighted energy in $0<\theta<\theta_I$ rises from $54.7\%$ to over $99.99\%$ in the high-diffusivity sector of the vertical case. In the low-diffusivity sector of the diagonal case it stays between $4.90\%$ and $4.91\%$. The corresponding penalties change from $279.19$ to $289.76$ and from $167.36$ to $167.43$, respectively. The theory gives the uniform bound $\gamma_h\le\beta C_E$ across mesh levels.

\begin{table}[htbp]
\centering\small
\caption{Relative physical energy errors on fitted interfaces. The upper panel compares the polynomial and enriched solutions on the vertical interface ($\lambda=3/4$). Each enriched solve uses $N_{\mathrm{dof}}+1$ coefficients. The middle panel lists the vertical-interface WG penalties and the selection results with $\nu=0.75$ withheld. The bottom panel lists the diagonal-interface penalties and errors for a prescribed physical exponent $1/2$ at the same dimensions.}\label{tab:transmission}
\setlength{\tabcolsep}{4pt}
\begin{tabular}{@{}r rrrr rrrr@{}}\toprule
 &\multicolumn{4}{c}{Quadratic WG}&\multicolumn{4}{c}{Conforming $P_2$}\\
\cmidrule(lr){2-5}\cmidrule(lr){6-9}
$L$ & $N_{\mathrm{dof}}$ & WG & WG $+\psi$ & Rate & $N_{\mathrm{dof}}$ & $P_2$ & $P_2+\psi$ & Rate\\\midrule
2 & 228 & 0.10593 & 0.055075 & -- & 33 & 0.12701 & 0.063087 & --\\
4 & 960 & 0.039455 & 0.015218 & 1.856 & 161 & 0.049269 & 0.017438 & 1.855\\
8 & 3,936 & 0.019233 & 0.0039223 & 1.956 & 705 & 0.024754 & 0.0044959 & 1.956\\
16 & 15,936 & 0.010926 & $9.8900\times10^{-4}$ & 1.988 & 2,945 & 0.014180 & 0.0011337 & 1.988\\
32 & 64,128 & 0.0064446 & $2.4775\times10^{-4}$ & 1.997 & 12,033 & 0.0083761 & $2.8406\times10^{-4}$ & 1.997\\
\bottomrule\end{tabular}
\par\medskip
\begin{tabular}{@{}lrrrrr@{}}\toprule
$L$ & 2 & 4 & 8 & 16 & 32\\\midrule
\multicolumn{6}{@{}l}{Vertical interface, exact exponent withheld}\\
$\sigma_\star$ & 279.19 & 281.32 & 286.12 & 288.93 & 289.76\\
Selected $\nu$ & 0.80 & 0.80 & 0.80 & 0.70 & 0.70\\
$e_{\rm phys}$ & 0.056490 & 0.016549 & 0.0049512 & 0.0017230 & $8.1649\times10^{-4}$\\\midrule
\multicolumn{6}{@{}l}{Diagonal interface, prescribed physical exponent $1/2$}\\
$\sigma_\star$ & 167.36 & 167.43 & 167.43 & 167.43 & 167.43\\
WG & 0.16791 & 0.10259 & 0.070689 & 0.049803 & 0.035192\\
WG $+\psi$ & 0.071880 & 0.019931 & 0.0051155 & 0.0012865 & $3.2193\times10^{-4}$\\
$P_2$ & 0.20586 & 0.12691 & 0.087444 & 0.061568 & 0.043491\\
$P_2+\psi$ & 0.085379 & 0.023662 & 0.0060948 & 0.0015363 & $3.8492\times10^{-4}$\\\bottomrule
\end{tabular}
\end{table}

\subsection{Local refinement with split faces and hanging vertices}\label{sec:splitfaces}
We use the constant-diffusion corner problem and prescribed profile $\psi$ from Section~\ref{sec:refinement}. Begin with squares of side $1/L$ on the L-shaped domain. Subdivide every square in the northeast quadrant into four equal squares and keep the original squares in the left half of the domain. The resulting $6L^2$ cells have $L$ hanging midpoints along $x=0$, $0<y<1$. For WG, the right side of each adjacent coarse cell consists of two mesh faces with independent linear traces. Joining the mean of the boundary vertices to each face gives a matching simplicial subdivision on both sides of the interface. There are finitely many fan shapes up to scaling, so this subdivision satisfies the mesh assumptions of Section~\ref{sec:background}.

For the conforming $P_1$ comparison in Table~\ref{tab:splitfaces}, split each square along its southwest--northeast diagonal. On each split side with endpoints $p_1,p_2$ and midpoint $p_{\rm mid}$, impose $v(p_{\rm mid})=[v(p_1)+v(p_2)]/2$ on the piecewise linear background. Both methods add the same continuous $\psi$ with a free amplitude and use the same volume forcing. The enriched conforming trace is the sum of the constrained polynomial trace and the actual trace of $\psi$. Before enrichment, the WG system has $42L^2-10L$ free coefficients, and the constrained nodal system has $6L^2-6L+1$.

\begin{table}[!htbp]
\centering\small
\caption{Locally refined corner problem with $L$ hanging midpoints. $N_{\mathrm{dof}}$ counts free polynomial coefficients after boundary conditions and, for $P_1$, midpoint constraints. The columns headed by $\psi$ add one coefficient. All errors are relative physical energy errors.}\label{tab:splitfaces}
\setlength{\tabcolsep}{4pt}
\begin{tabular}{@{}r rrrr rrr@{}}\toprule
 &\multicolumn{4}{c}{WG with split faces}&\multicolumn{3}{c}{Constrained $P_1$}\\
\cmidrule(lr){2-5}\cmidrule(l){6-8}
$L$ & $N_{\mathrm{dof}}$ & $\sigma_\star$ & Polynomial & With $\psi$ & $N_{\mathrm{dof}}$ & Polynomial & With $\psi$\\\midrule
4 & 632 & 163.81 & 0.792185 & 0.568866 & 73 & 0.515488 & 0.281914\\
8 & 2,608 & 169.27 & 0.571039 & 0.381649 & 337 & 0.273085 & 0.145084\\
16 & 10,592 & 172.01 & 0.359028 & 0.208832 & 1,441 & 0.149191 & 0.073110\\
32 & 42,688 & 173.41 & 0.205764 & 0.105308 & 5,953 & 0.083650 & 0.036637\\
\bottomrule\end{tabular}
\end{table}
\FloatBarrier

On the finest mesh in Table~\ref{tab:splitfaces}, the added mode reduces the WG error from $0.205764$ to $0.105308$ while retaining the independent subface traces. The enriched WG solve uses $42{,}689$ coefficients. Constrained $P_1$ gives $0.083650$ with $5{,}953$ coefficients before enrichment and $0.036637$ with $5{,}954$ after adding the same profile. As $L$ doubles from $4$ to $32$, the successive rates $\log_2(e_L/e_{2L})$ for the enriched physical energy errors are $0.576$, $0.870$ and $0.988$ for WG and $0.958$, $0.989$ and $0.997$ for constrained $P_1$.

The unknowns on the two half-edges represent solution traces. To examine the cell balance, write $U_m=v_m+\iota z_m$. In the flux of Mu, Wang and Ye \cite[Eq.~(32)]{MWY}, the vector projection satisfies $\Pi_T(a_T\nabla_w v_m)=a_T\nabla_w v_m$ because $a_T$ is constant on $T$. The outward normal flux on each face $F$ is therefore
\[
 \phi_{T,F}(v_m)
   =\left[-(a_T\nabla_w v_m)\cdot\nu_T
       +\sigma h_T^{-1}(v_{m,0}-v_{m,b})\right]\big|_F .
\]
Take the WG test $w_T=(\mathbf 1_T,0)$, where the cell indicator is extended by zero to $\Omega$. Testing \eqref{eq:wg} with $a_T\nabla_wv_m$ and integrating by parts gives $\sum_{F\subset\partial T}\int_F\phi_{T,F}(v_m)\dd s=a_{\mathrm{WG}}^\sigma(v_m,w_T)$. The enriched cell equation \eqref{eq:blocks} and the cell moments $Q_0\Eh w_T=\mathbf 1_T$ therefore yield
\begin{equation}\label{eq:cellbalance}
 \sum_{F\subset\partial T}\int_F\phi_{T,F}(v_m)\dd s-\int_T f\dd x
   =(f-Q_0 f,\Eh w_T-\mathbf 1_T)-b(z_m,\Eh w_T).
\end{equation}
The first pairing is over $\Omega$, with $Q_0 f$ evaluated cellwise. For $m=0$ and $f|_S\in\Pol_k(S)$ on every cell $S$, both terms vanish and the total outward flux equals the cell source. In the enriched solve, the two terms give the contributions of load reconstruction and profile coupling to the difference.

Table~\ref{tab:flux} compares integrated fluxes on the same polygon cells for both enriched solutions. Here $a=I$. For the WG field with amplitude $c_{\rm WG}$ we use $q^{\rm WG}_{T,F}=\phi_{T,F}(v_1)-c_{\rm WG}\nabla\psi\cdot\nu_T$. For the conforming field $p_h+c_{\rm FE}\psi$ we use $q^{\rm FE}_{T,F}=-\nabla(p_h+c_{\rm FE}\psi)|_{T}\cdot\nu_T$, taking the gradient from the triangle inside $T$ along $F$. Each amplitude comes from its own Galerkin solve. With $I_T=\int_T f\dd x$ and $M_L=\max_T|I_T|$, the normalized maximum cell imbalance is
\[
 \eta(q)=\frac{1}{M_L}\max_T\left|\sum_{F\subset\partial T}\int_F q_{T,F}\dd s-I_T\right|.
\]
Each split side contributes its two half-faces. The conforming flux is the element gradient before any flux recovery.

\begin{table}[htbp]\centering
\caption{Maximum integrated cell imbalance for the enriched solutions on the locally refined grids of Table~\ref{tab:splitfaces}. $M_L$ is the largest absolute cell source integral.}\label{tab:flux}
\begin{tabular}{rccc}\toprule
$L$ & $M_L$ & $\eta(q^{\rm WG})$ & $\eta(q^{\rm FE})$\\\midrule
4  & 0.728532 & $1.090\times10^{-3}$ & 0.4526\\
8  & 0.205929 & $2.678\times10^{-4}$ & 0.4955\\
16 & 0.060646 & $5.974\times10^{-5}$ & 0.4667\\
32 & 0.015295 & $1.517\times10^{-5}$ & 0.4866\\
\bottomrule\end{tabular}
\end{table}

The WG imbalance decreases on each refinement. The conforming solution has the smaller physical energy error at all four levels. As a separate check of \eqref{eq:cellbalance}, we integrated the polynomial WG face flux, the cell source and both correction terms independently. The largest absolute discrepancy between the two sides, divided by $M_L$, was below $10^{-11}$ on every mesh at quadrature orders 24 and 32. Both orders give the entries in Table~\ref{tab:flux} to the displayed precision.

\FloatBarrier
\section{Conclusions}
The correction $\jumpform=\sigma J-\highform$ gives the original polynomial WG stiffness at each fixed penalty. The quotient eigenvalue $\gamma_h$ locates the sharp threshold for positivity on the full space. Above that threshold, Galerkin projection gives best approximation in the induced norm, and the squared residual score gives the exact decrease in squared error from one independent direction.

With the global quotient available, we recommend $\sigma_\star=1+\gamma_h$ to retain the unit lower bound $\jumpform\ge J$. The half-unit margin gave nearly the same physical errors on the five triangular grids and required the same eigenvalue calculation. When a global eigensolve is impractical, the patchwise upper bound gives the unit lower bound through local eigenproblems. The larger patch penalty raised the full-dictionary error in the 24-triangle test.

For an existing WG solver, the reconstruction and penalty calculation are setup costs. Changing the penalty requires a new numerical factorization. Once the polynomial block at the chosen penalty has been factored and the reconstructed load solved, each trial of a single prescribed profile uses one mixed-column solve and a scalar Schur complement. On the finest locally refined mesh, adding one coefficient reduces the WG physical energy error from $0.205764$ to $0.105308$ while retaining both polynomial traces on each split side. On that mesh and the finest polygon fan, unenriched conforming elements achieve smaller errors with fewer coefficients. On all four locally refined meshes the enriched WG numerical flux has a smaller maximum cell imbalance than the unprocessed gradient flux of enriched conforming $P_1$. On triangles, both corner grading alone and uniform refinement with enrichment give energy rates approaching one.

The fitted diagonal interface shows what one coefficient can add to a fixed WG discretization. The square-root profile has finite energy and lies outside the face regularity in (A3), yet both enriched quadratic discretizations approach order two in energy. On the vertical interface the exact exponent ranks first when it belongs to the supplied list. Identifying an exponent outside a prescribed family remains open. The analysis also covers polyhedra under the mesh assumptions of Section~\ref{sec:background}. A three-dimensional numerical test remains to be done, particularly for an edge singularity whose amplitude varies along the edge and must itself be approximated.

\subsection*{Funding}
This research received no external funding.


\begin{thebibliography}{99}
\bibitem{WY} J. Wang and X. Ye, A weak Galerkin finite element method for second-order elliptic problems, \emph{J. Comput. Appl. Math.} \textbf{241} (2013), 103--115. \href{https://doi.org/10.1016/j.cam.2012.10.003}{doi:10.1016/j.cam.2012.10.003}.
\bibitem{MWY} L. Mu, J. Wang, and X. Ye, Weak Galerkin finite element methods on polytopal meshes, \emph{Int. J. Numer. Anal. Model.} \textbf{12} (2015), 31--53. \href{https://doi.org/10.4208/ijnam.v12.n1.p31}{doi:10.4208/ijnam.v12.n1.p31}.
\bibitem{gWG} D. Li, C. Wang, J. Wang, and X. Ye, Generalized weak Galerkin finite element methods for second order elliptic problems, \emph{J. Comput. Appl. Math.} \textbf{445} (2024), 115833. \href{https://doi.org/10.1016/j.cam.2024.115833}{doi:10.1016/j.cam.2024.115833}.
\bibitem{MWYZinterface} L. Mu, J. Wang, X. Ye, and S. Zhao, A new weak Galerkin finite element method for elliptic interface problems, \emph{J. Comput. Phys.} \textbf{325} (2016), 157--173. \href{https://doi.org/10.1016/j.jcp.2016.08.024}{doi:10.1016/j.jcp.2016.08.024}.
\bibitem{YZfree} X. Ye and S. Zhang, A stabilizer-free weak Galerkin finite element method on polytopal meshes, \emph{J. Comput. Appl. Math.} \textbf{371} (2020), article 112699. \href{https://doi.org/10.1016/j.cam.2019.112699}{doi:10.1016/j.cam.2019.112699}.
\bibitem{Yemm} L. Yemm, Design and analysis of the extended hybrid high-order method for the Poisson problem, \emph{Adv. Comput. Math.} \textbf{48} (2022), article 45. \href{https://doi.org/10.1007/s10444-022-09958-y}{doi:10.1007/s10444-022-09958-y}.
\bibitem{BCMS} E. Benvenuti, A. Chiozzi, G. Manzini, and N. Sukumar, Extended virtual element method for the Laplace problem with singularities and discontinuities, \emph{Comput. Methods Appl. Mech. Engrg.} \textbf{356} (2019), 571--597. \href{https://doi.org/10.1016/j.cma.2019.07.028}{doi:10.1016/j.cma.2019.07.028}.
\bibitem{AM} E. Artioli and L. Mascotto, Enrichment of the nonconforming virtual element method with singular functions, \emph{Comput. Methods Appl. Mech. Engrg.} \textbf{385} (2021), article 114024. \href{https://doi.org/10.1016/j.cma.2021.114024}{doi:10.1016/j.cma.2021.114024}.
\bibitem{DMY} J. Droniou, G. Manzini, and L. Yemm, The eXtended virtual element method for elliptic \mbox{problems} with weakly singular solutions, \emph{Comput. Methods Appl. Mech. Engrg.} \textbf{429} (2024), article 117129. \href{https://doi.org/10.1016/j.cma.2024.117129}{doi:10.1016/\allowbreak j.cma.2024.117129}.
\bibitem{VZ} A. Veeser and P. Zanotti, Quasi-optimal nonconforming methods for symmetric elliptic problems. I---Abstract theory, \emph{SIAM J. Numer. Anal.} \textbf{56} (2018), 1621--1642. \href{https://doi.org/10.1137/17M1116362}{doi:10.1137/17M1116362}.
\bibitem{EZ} A. Ern and P. Zanotti, A quasi-optimal variant of the hybrid high-order method for elliptic partial differential equations with $H^{-1}$ loads, \emph{IMA J. Numer. Anal.} \textbf{40} (2020), 2163--2188. \href{https://doi.org/10.1093/imanum/drz057}{doi:10.1093/imanum/drz057}.
\bibitem{CN} C. Carstensen and N. Nataraj, A priori and a posteriori error analysis of the Crouzeix--Raviart and Morley FEM with original and modified right-hand sides, \emph{Comput. Methods Appl. Math.} \textbf{21} (2021), 289--315. \href{https://doi.org/10.1515/cmam-2021-0029}{doi:10.1515/cmam-2021-0029}.
\bibitem{Grisvard} P. Grisvard, \emph{Elliptic Problems in Nonsmooth Domains}, Classics in Applied Mathematics, vol.~69, SIAM, Philadelphia, 2011. \href{https://doi.org/10.1137/1.9781611972030}{doi:10.1137/1.9781611972030}.
\bibitem{FGW} G.~J. Fix, S. Gulati, and G.~I. Wakoff, On the use of singular functions with finite element approximations, \emph{J. Comput. Phys.} \textbf{13} (1973), 209--228. \href{https://doi.org/10.1016/0021-9991(73)90023-5}{doi:10.1016/0021-9991(73)90023-5}.
\bibitem{SF} G. Strang and G.~J. Fix, \emph{An Analysis of the Finite Element Method}, Prentice-Hall, Englewood Cliffs, NJ, 1973.
\bibitem{BKP} I. Babu\v{s}ka, R.~B. Kellogg, and J. Pitk\"aranta, Direct and inverse error estimates for finite elements with mesh refinements, \emph{Numer. Math.} \textbf{33} (1979), 447--471. \href{https://doi.org/10.1007/BF01399326}{doi:10.1007/BF01399326}.
\bibitem{CaiKim} Z. Cai and S. Kim, A finite element method using singular functions for the Poisson equation: corner singularities, \emph{SIAM J. Numer. Anal.} \textbf{39} (2001), 286--299. \href{https://doi.org/10.1137/S0036142999355945}{doi:10.1137/S0036142999355945}.
\bibitem{MB} J.~M. Melenk and I. Babu\v{s}ka, The partition of unity finite element method: basic theory and applications, \emph{Comput. Methods Appl. Mech. Engrg.} \textbf{139} (1996), 289--314. \href{https://doi.org/10.1016/S0045-7825(96)01087-0}{doi:10.1016/S0045-7825(96)01087-0}.
\bibitem{SBC} T. Strouboulis, I. Babu\v{s}ka, and K. Copps, The design and analysis of the Generalized Finite Element Method, \emph{Comput. Methods Appl. Mech. Engrg.} \textbf{181} (2000), 43--69. \href{https://doi.org/10.1016/S0045-7825(99)00072-9}{doi:10.1016/S0045-7825(99)00072-9}.
\bibitem{AD} M. Ainsworth and J. Dong, Galerkin neural networks: A framework for approximating variational equations with error control, \emph{SIAM J. Sci. Comput.} \textbf{43} (2021), A2474--A2501. \href{https://doi.org/10.1137/20M1366587}{doi:10.1137/20M1366587}.
\bibitem{ADsing} M. Ainsworth and J. Dong, Extended Galerkin neural network approximation of singular variational problems with error control, \emph{SIAM J. Sci. Comput.} \textbf{47} (2025), no.~3, C738--C768. \href{https://doi.org/10.1137/24M1658279}{doi:10.1137/24M1658279}.
\bibitem{Siegel} J. W. Siegel, Q. Hong, X. Jin, W. Hao, and J. Xu, Greedy training algorithms for neural networks and applications to PDEs, \emph{J. Comput. Phys.} \textbf{484} (2023), article 112084. \href{https://doi.org/10.1016/j.jcp.2023.112084}{doi:10.1016/j.jcp.2023.112084}.
\bibitem{Barucq} H. Barucq, M. Duprez, F. Faucher, E. Franck, F. Lecourtier, V. Lleras, V. Michel-Dansac, and N. Victorion, Enriching continuous Lagrange finite element approximation spaces using neural networks, \emph{ESAIM Math. Model. Numer. Anal.} \textbf{60} (2026), no.~2, 541--603. \href{https://doi.org/10.1051/m2an/2026007}{doi:10.1051/m2an/2026007}.
\bibitem{CLC} X. Chen, Y. Luo, and J. Chen, A PINN-enriched finite element method for linear elliptic problems, \emph{J. Comput. Appl. Math.} \textbf{481} (2026), article 117268. \href{https://doi.org/10.1016/j.cam.2025.117268}{doi:10.1016/j.cam.2025.117268}.
\bibitem{OLS} L. Rebollo-Neira and D. Lowe, Optimized orthogonal matching pursuit approach, \emph{IEEE Signal Process. Lett.} \textbf{9} (2002), 137--140. \href{https://doi.org/10.1109/LSP.2002.1001652}{doi:10.1109/LSP.2002.1001652}.
\bibitem{WangPreprint} C. Wang, A neural-enhanced weak Galerkin method for second-order elliptic problems with low-regularity solutions, preprint (2026). \href{https://arxiv.org/abs/2604.06115v1}{arXiv:2604.06115v1}.
\bibitem{EG} A. Ern and J.-L. Guermond, Finite element quasi-interpolation and best approximation, \emph{ESAIM Math. Model. Numer. Anal.} \textbf{51} (2017), 1367--1385. \href{https://doi.org/10.1051/m2an/2016066}{doi:10.1051/m2an/2016066}.
\bibitem{YAS} Z. Yosibash, R. Actis, and B. Szab\'o, Extracting edge flux intensity functions for the Laplacian, \emph{Int. J. Numer. Methods Eng.} \textbf{53} (2002), 225--242. \href{https://doi.org/10.1002/nme.399}{doi:10.1002/nme.399}.
\bibitem{He} J. He, L. Li, J. Xu, and C. Zheng, ReLU deep neural networks and linear finite elements, \emph{J. Comput. Math.} \textbf{38} (2020), 502--527. \href{https://doi.org/10.4208/jcm.1901-m2018-0160}{doi:10.4208/jcm.1901-m2018-0160}.
\bibitem{BD} H. Blum and M. Dobrowolski, On finite element methods for elliptic equations on domains with corners, \emph{Computing} \textbf{28} (1982), 53--63. \href{https://doi.org/10.1007/BF02237995}{doi:10.1007/BF02237995}.
\bibitem{KelloggI} R.~B. Kellogg, On the Poisson equation with intersecting interfaces, \emph{Appl. Anal.} \textbf{4} (1974), no.~2, 101--129. \href{https://doi.org/10.1080/00036817408839086}{doi:10.1080/00036817408839086}.
\end{thebibliography}
\end{document}